\documentclass[a4paper,11pt]{article}
\usepackage[english]{babel}
\usepackage{graphicx}

\usepackage{amsxtra}
\usepackage{amssymb}
\usepackage{amsmath}
\usepackage{amsthm}
\usepackage{color}
\usepackage{fancyhdr}

\usepackage{hyperref}

\newtheorem{prop}{Proposition}[section]

\newtheorem{lemma}[prop]{Lemma}
\newtheorem{theo}[prop]{Theorem}
\newtheorem{coro}[prop]{Corollary}

\numberwithin{equation}{section}

\theoremstyle{remark}
\newtheorem{rmq}{Remark}

\newcommand{\di}{\displaystyle}
\newcommand{\mH}{\mathcal{H}}

\newcommand{\R}{\mathbb{R}}
\newcommand{\N}{\mathbb{N}}

\title{Existence of multi-traveling waves in capillary fluids}
\author{Corentin Audiard
\footnote{Sorbonne Universit\'es, UPMC Univ Paris 06, UMR 7598, Laboratoire Jacques-Louis Lions, 
F-75005, Paris, France }
\footnote{CNRS, UMR 7598, Laboratoire Jacques-Louis Lions, F-75005, Paris, France}}

\begin{document}
\maketitle

\begin{abstract}
We prove the existence of multi-soliton and kink-multi-soliton solutions of the 
Euler-Korteweg system in dimension one. Such solutions behaves asymptotically in time 
like several traveling waves far away from each other.
A kink is a traveling wave with different 
limits at $\pm\infty$.
The main  assumption is the linear stability of the solitons, and we prove that this 
assumption is satisfied at least in the transonic limit.
The proof relies on a classical approach based on energy estimates 
and a compactness argument. 
\end{abstract}
\tableofcontents
\section{Introduction}
\paragraph{The Euler-Korteweg model}
The Euler-Korteweg equations read 
\begin{equation}\label{EK}
\left\{
\begin{array}{lll}
\partial_t\rho+\partial_x(\rho v)&=&0,\\
\partial_t v+v\partial_xv+g'(\rho)\partial_x\rho&=&\partial_x\bigg(K(\rho)\partial_x^2\rho
+\frac{1}{2}K'(\rho)(\partial_x\rho)^2\bigg),
\end{array}\right.
(x,t)\in\R\times \R^+.
\end{equation}
They are a modification of the usual Euler equations that model capillary forces in non viscous 
fluids. The function $K(\rho)$ is supposed to be smooth $\R^{+*}\to \R^{+*}$. In some 
relevant cases it is not bounded near $0$, in particular for $K=1/\rho$ there exists a
change of variable, the Madelung transform, that converts at least formally solutions of 
\eqref{EK} into solutions of the nonlinear Schr\"odinger equation (for details on this interesting 
feature see the review article \cite{CDS}).\\
There is a formally conserved energy 
\begin{equation*}
 H[\rho,v]=\int_{\R}\frac{1}{2}\big(\rho v^2+K(\rho)(\partial_x\rho)^2\big)+G(\rho)dx,
\end{equation*}
where $G$ is a primitive of $g$, and under appropriate functional settings, 
denoting $\delta H$ the variational derivative of $H$, $\di V=\begin{pmatrix}\rho\\ v\end{pmatrix}$
\eqref{EK} can be viewed as a hamiltonian system 
\begin{equation}\label{EKHam}
\partial_t V
=J\partial_x \delta H[V],\text{ with }
J=\begin{pmatrix} 0 & -1 \\ -1 & 0\end{pmatrix}.
\end{equation}
\eqref{EK} also has a formally conserved momentum $P(\rho,v)=\int_{\R}\rho v$, whose conservation 
is related to the identity $\delta P[V]=-JV$. Although formal 
these identities are used at least as notations in this article.\\
Due to the intricate quasilinear nature of \eqref{EK}, only local well-posedness 
was obtained so far in dimension one and, even for data close to 
a constant state, global well-posedness is an open problem
 (on well-posedness and stability in larger 
dimension, see also \cite{Benzoni1}, \cite{AudHasp2}, \cite{GiesTza}). \\
It was proved in \cite{BDD2} by ODE technics that \eqref{EK} admits traveling waves as 
solutions, namely solutions of the form $(\rho,v)(x-ct)$. There exists two classes 
of traveling waves: those such that $\di \lim_{+\infty}(\rho,v)(z)\neq 
\lim_{-\infty}(\rho,v)(z)$ are labelled as kinks, while solitons satisfy 
$\di \lim_{+\infty}(\rho,v)(z)=\lim_{-\infty}(\rho,v)(z)$. No quantity is assumed to be 
zero at infinity. \\
Both types of traveling waves are physically relevant, especially kinks are supposed to model 
phase transition in capillary fluid (e.g. liquid to vapor).
While kinks are known to be always stable, solitons are not and a (conditional) 
stability criterion in the spirit of \cite{GSS} was derived in \cite{BDD2}.\\
This article is devoted to a related, yet different issue : the 
existence of multi-traveling waves, i.e. solutions that decouple as $t\to +\infty$ 
to a sum of traveling waves. 
\paragraph{Multi-traveling waves in the litterature}
The existence of multiple traveling waves is now a classical topic. While the first 
examples came from the field of integrable equations (for example, 
see the pioneering work of Zakharov-Shabat \cite{ZakShab}), flexible and powerful methods 
have since been developed to tackle non-integrable equation. In particular considerable 
progress was achieved for the KdV and nonlinear Schr\"odinger equations over the 
last twenty years.\\ In the framework of the 
nonlinear Schr\"odinger equation we refer to the work of Martel, Merle and coauthors 
\cite{MartelMerle06, MartelMerle11}, in particular based the use of 
modulation parameters and a compactness argument, 
see also Le Coz and Tsai \cite{LeCozTsai14} for an approach based on dispersive estimates.
To the best of our knowledge, the inclusion of a kink in the asymptotic profile is a rather 
rare feature in the field of multiple traveling waves, to the noticeable exception of the work 
of Le Coz-Tsai \cite{LeCozTsai14}, see also \cite{LinTsai17} in higher dimension.\\
All those results share the fact that they are more conveniently applied to equations that have 
a ``good'' well-posedness theory available (existence of global solutions in not too restrictive 
spaces). As such, their adaptation to quasilinear systems like the Euler-Korteweg system 
raises some difficulties. 
To explain it roughly, a key step of the compactness argument of Martel-Merle requires 
the existence of solutions for $t\in \R^+$, while the well-posedness 
theory for the Euler-Korteweg system only allows existence in finite time. \\
In the context of the water waves, this difficulty was overcome by Ming-Rousset-Tzvetkov 
\cite{MRT15} with the construction of global \emph{approximate} solutions with some 
fast decay in time on the approximation error. 
This is also, to some extent, the approach that we follow here.

\paragraph{The traveling waves}
A short description of the construction of traveling waves is provided in the appendix, for more 
details we refer to \cite{BDD2}. Their main features are the following:
\begin{itemize}
\item A traveling wave is a solution of \eqref{EK} of the form $V(x-ct)$, $c$ is its speed.
 \item All the traveling waves that we consider are smooth and bounded. Their 
 derivatives are exponentially decreasing 
 at $\pm\infty$. Consequently all traveling waves have limits at $\pm\infty$.
 \item Kinks are travelling waves that have different endstates 
 at $\pm\infty$, say $(\rho_\pm,u_\pm)$.
 The function $\rho$ is monotonous.
 \item Solitons are traveling waves with same endstates at $\pm\infty$. $\rho$ changes monotony 
 only once, when it reaches its unique extremum. In the appendix we only deal with the case  where this extremum is a minimum. By analogy with the Schr\"odinger 
 equation (see e.g. \cite{Barashenkov}), we label such solutions 
 bubbles (note that for fluids such solutions correspond to a negative bump in the 
 density, therefore the word bubble is consistent).
 \item For a fixed endstate solitons can be smoothly parametrized by 
their speed. 
\end{itemize}

\paragraph{Main result}
Let $(V^{c_j})_{1\leq j\leq n}$ be travelling waves with ordered speeds $c_j<c_{j+1}$. 
We assume that all $V^{c_j}$ are stable, and we consider one of the following two cases: 
\begin{itemize}
 \item $V^{c_1}$ is a kink, and $V^{c_j}$ are solitons with 
 $\di \forall\,j\geq 2,\ \lim_{\pm\infty}V^{c_j}=\lim_{+\infty}V^{c_1}$.
 \item $V^{c_j}$ are all solitons with the same endstate.
\end{itemize}
For $V^c$ a soliton, we denote $\di U^c=V^c-
\begin{pmatrix}\rho_+\\v_+\end{pmatrix}$, and 
define the rescaled momentum 
\begin{equation}
 \label{momentum}
P(V^c)=P(\rho^c,v^c)=\int_{\R} (\rho^c-\rho_+)(v^c-v_+)dx.
\end{equation}
We assume that the solitons are stable in the following sense
\begin{equation*}
\text{Stability condition: } \frac{d}{dc}\int (\rho^c-\rho_+)(v^c-v_+)dx<0.
\end{equation*}
We refer to the appendix \ref{appendix} for the proof that our
conditions can be met, where we also show that this stability criterion coincides with 
the one derived in \cite{BDD2}.\\
We define the multi-soliton 
$$S(x,t)=V^{c_1}(x-c_1t)+\sum_{k=2}^nU^{c_k}(x-c_kt-\sum_{j=2}^kA_j),$$ $A_j\geq A$ a 
large constant to choose later.
Our main result is that there exists a solution which converges to $S$ as $t\to \infty$ (see section \ref{not} for the definition 
of $\mathcal{H}^n$).
\begin{theo}\label{mainresult}
For $A_0$ large enough, $A\geq A_0$, and $n\geq 3$, 
there exists a global  solution of \eqref{EK} such that 
$V-S\in C(\R^+,\mathcal{H}^{2n})$ and 
\begin{equation*}
\lim_{t\to \infty} \|V(t)-S(t)\|_{\mathcal{H}^{2n}}\to 0.
\end{equation*}
\end{theo}
\begin{rmq}
It may be tempting to think that theorem \ref{mainresult} hints
towards the stability of multi-solitons. This is not 
correct as the solution constructed is quite peculiar: it is a pure 
soliton solution with no dispersive part. 
For NLS multi-solitons 
have been constructed in cases where each soliton is 
unstable\cite{MartelMerle11, Combet14}.\\
Note however that in the case of the Gross-Pitaevskii equation, whose hydrodynamics
formulation is a special case of \eqref{EK} with $K=1/\rho$, $g=\rho-1$, nonlinear 
stability of multi-solitons was obtained by B\'ethuel-Gravejat-Smets \cite{BGSchain}.
It is expectable that a similar result holds (at least in some regime) for 
\eqref{EK}, however, due to the lack of global well-posedness, going beyond 
\emph{conditional stability}, that is stability until blow-up, requires significant 
new ideas.
\end{rmq}
\begin{rmq}
It is apparent from the proof that multiple traveling waves can be constructed in more complicated configurations, such as kink-soliton-kink, soliton-kink-kink etc. 
We chose not to aim at such results to keep a reasonably simple proof, and because 
configurations with multiple kinks and stable solitons might require very exotic presssure laws to 
exist.
\end{rmq}

%
\paragraph{Scheme of proof} The key is to construct an approximate solution $V^a$ to 
\eqref{EK} that satisfies 
\begin{equation*}
 \partial_t V^a-J\partial_x\delta H[V^a]=f^a,
\end{equation*}
which is defined globally, converges as $t\to \infty$ to the multi-soliton,
and such that the error term $f^a$ decays rapidly in time.
Once $V^a$ is constructed, we use the local well-posedness theory with some improved energy 
estimates to construct a sequence of exact solutions $V^k$ close to $V^a$,
defined on $[0,k]$ with $V^k(k)=U^a(k)$.
A compactness argument then provides a global solution of the Euler-Korteweg system which converges
at $t\to \infty$ to the multi-soliton.\\
The construction of $V^a$ is quite intricate, it requires fine estimates on the flow generated by 
the linearized operator $J\partial_x\delta^2H[S]$, building upon a spectral decomposition of 
$\delta^2H[V^{c_j}]-c_j\delta^P$. Once these estimates are proved, the approximate solution is constructed by a Newton iteration method applied to the problem 
$\partial_tV-J\partial_x\delta H[V]=0$.

\paragraph{Plan of the article} In section \ref{not} we define some notations and 
functional settings. The energy estimate for \eqref{EK} are proved in section \ref{secenergie}.\\
Section \ref{seclin} is the core of the article. We first give a convenient spectral 
decomposition of the operator $\delta^2E[V^{c_j}]-c_j\delta^2P$. We deduce some 
estimates on the flow of $J\partial_x\delta^2H$ that are not useful for this paper, but contain 
most of the ideas for the much more technical estimates on the flow of $J\partial_x\delta^2H[S]$.\\
With these estimates at hand we construct in section \ref{secnewton} an approximate solution by 
following Newton's iteration method. The compactness argument that provides the multi-soliton
solution is detailed in section \ref{secpreuve}. \\
Finally, as the existence of a ``kink-stable solitons'' configuration is not obvious, we prove it 
in the appendix. The appendix is also used to recall how kinks and solitons for \eqref{EK} are 
constructed.

\paragraph{Acknowledgement}
The author was supported by the French National Research Agency project NABUCO,
grant ANR-17-CE40-0025.\\
The author thanks Stefan Le Coz for pointing out some useful references.

\section{Notations, functional spaces}\label{not}
\textbf{Reference state of a solution} Any solution $V$ of \eqref{EK} that we consider is of the form 
\begin{equation}\label{notref}
V=V_{\text{ref}}+U,
\end{equation}
where $U$ vanishes at infinity, $V_{\text{ref}}$ is a reference state which is a 
smooth function with finite limit at $\pm\infty$, and for any $k+j\geq 1$, 
$\partial_x^k\partial_t^jV_{\text{ref}}$ decays exponentially at $\pm \infty$. \\
The notation $V=V_{\text{ref}}+U$ will be used without explanation
when the context is clear, in particular for a soliton of endstate 
$(\rho_+,v_+)$ we always take $V_{\text{ref}}=(\rho_+,v_+)^t$. 
If any sub/superscript is present we denote
$V^a=V_{\text{ref}}^a+U^a,\ V_j=V_{\text{ref},j}+U_j$ etc. 
\\
We always denote $V=\begin{pmatrix}\rho\\ v \end{pmatrix}$, $U=\begin{pmatrix}r\\ u \end{pmatrix}$, and similarly for $V^a,U^a$...
\paragraph{Symbols and conventions of computation}
The constant $C$ in inequalities $A\leq CB$ changes from line to 
line. Depending on the context, they are allowed to depend on some 
quantities, but for conciseness this dependency 
is not explicited. For example when proving $A\leq 
C(\|u\|_{\infty})B$, we write freely $|uv|\leq C|v|$.\\
The inequality $A\lesssim B$ means $A\leq CB$ for some constant
$C>0$, where the previous rule applies to $C$.\\
The $L^2$ scalar product for real vector valued functions is denoted 
$\langle\,\cdot,\,\cdot\rangle$.
\paragraph{Sobolev spaces} Even for functions of one variable, we use 
the notation $u'=\partial_xu$. 
$H^n$ is the usual $L^2$ based Sobolev space 
$$H^n=\{u\in \mathcal{S}':\ \forall\,0\leq k\leq n,\ 
\partial_x^ku\in L^2\},\ \|u\|_{H^n}^2=\sum_0^n\|\partial_x^ku\|_{L^2}^2.$$ 
We denote $\mathcal{C}^n_b$ the set of $n$ times differentiable 
functions that are bounded as well as their derivatives.
For a vector valued distribution $U=\begin{pmatrix}r\\u\end{pmatrix}$, we also define 
\begin{equation*}
\|U\|_{\mathcal{H}^n}^2=\|r\|_{H^{n+1}}^2+\|u\|_{H^n}^2,\text{ and }
\|U\|_{X^n}=\|U\|_{\mathcal{H}^{n+1}}
+\|\partial_tU\|_{\mathcal{H}^n}.
\end{equation*}
We have the interpolation property 
\begin{equation*}
\forall\, 0\leq k\leq n,\ \|u\|_{H^k}\leq \|u\|_{L^2}^{1-k/n}\|u\|_{H^n}^{k/n},
\end{equation*}
the continuous embedding $H^n\subset \mathcal{C}_b^{n-1}$. For $n\geq 1$, $H^n$
is a Banach algebra.\\
The following composition composition rules hold for 
$a\in \mathcal{C}_b^n+H^n,\ u\in H^n$ and $F$ smooth
on $\text{Im}(a),\ \text{Im}(a+u)$:
\begin{eqnarray}\label{calcsobo}
\|F(a+u)-F(a)\|_{H^n}\leq C(\|a\|_{\mathcal{C}^n_b+H^n}+\|u\|_{H^n})\|u\|_{H^n},\\
\nonumber
\text{ in particular if }F(0)=0,\  
\|F(u)\|_{H^n}\leq C(\|u\|_{H^n})\|u\|_{H^n}.
\end{eqnarray}
A similar second order rule holds
\begin{equation}\label{calcsobo2}
\|F(a+u)-F(a)-uF'(a)\|_{H^n}\leq C(\|a\|_{\mathcal{C}^n_b+H^n}
+\|u\|_{H^n})\|u\|_{H^n}^2.
\end{equation}
Both are consequences of a combination of the Faa Di Bruno formula, Sobolev's embedding, H\"older's inequality and Taylor's formula.

\section{Energy estimates}\label{secenergie}
An essential step is to bound the distance between an exact solution and a smoother approximate solution
$V^a=(\rho^a,v^a)$ satisfying 
\begin{equation*}
 \partial_tV^a=J\partial_x\delta H[V^a]+f^a, \text{ for some remainder }f^a.
\end{equation*}
Due to the quasi-linear nature of the system the flow map is (probably) not Lipschitz even in high regularity Sobolev 
spaces, nevertheless Lipschitz bounds with harmless loss of derivatives on $V^a$ can be obtained.\\
Energy estimates were obtained by Benzoni et al \cite{BDD} thanks to a change of variable (initially due to 
F. Coquel), and this section is actually more or less contained in \cite{BDD}.
Let us shortly describe the argument : if $(\rho,v)$ is a smooth
solution of \eqref{EK} without vacuum, $n\geq 2$, set 
$w=\sqrt{K/\rho}\partial_x \rho$, and $z=v+iw$. Then $z$ satisfies 
\begin{equation}
 \partial_tz+v\partial_x z+iw\partial_x z+i\partial_x\big(a\partial_x z\big)+g'(\rho)\partial_x\rho=0,
\end{equation}
with $a(\rho)=\sqrt{\rho K}$. This equation has a nice structure :
$i\partial_xa\partial_x$ is antisymetric, $v\partial_x$ too 
up to zero order terms, $g'\partial_x\rho$ is of order zero since $w$ is a derivative of $\rho$. The only bad term $iw\partial_x z$ is dealt with thanks to a gauge method.\\
A few preliminary notations : for $V=(\rho,v),$ solution of \eqref{EK} and $V^a=(\rho^a,v^a)$ an approximate solution, we 
denote $z$ and $z^a$ the associated new variables. We assume that $V^a=V_{\text{ref}}+U^a$ 
and $V=V_{\text{ref}}+U$  (same reference state) so that $V-V^a=U-U^a$.\\
Generically for $F$ a function of $\rho$ we denote  $\Delta F=F(\rho)
-F(\rho^a)$, for $F$ a function of $v$, $\Delta F=F(v)-F(v^a)$ etc. The gauge function of order $n$ is
$\varphi_n(\rho):=a^{n/2}\sqrt{\rho}$ and the modified norm 
\begin{equation*}
\widetilde{\|}\Delta V\|_{\mathcal{H}^{2n}}:=\|\Delta \rho\|_{L^2}+\|\sqrt{\rho}\Delta z\|_{L^2}
+\|\varphi_n\partial_x^{2n}\Delta z\|_{L^2}.
\end{equation*}
This notation is quite incorrect as the ``norm'' depends on $V$ in a nonlinear way. 
Nevertheless, using the computation rules \ref{calcsobo} and with constants depending 
continuously on $\|V\|_{\mathcal{H}^{2n}}+\|V^a\|_{\mathcal{H}^{2n}}
+\big\|\rho + \frac{1}{\rho}\big\|_{L^\infty}+
\big\|\rho^a + \frac{1}{\rho^a}\big\|_{L^\infty}$,
we have 
$\widetilde{\|}\Delta V\|_{\mathcal{H}^{2n}}\sim \|\Delta \rho\|_{L^2}+\|\Delta z\|_{H^{2n}}$, 
and 
$\|(\Delta \rho,\Delta v)\|_{\mathcal{H}^{2n}}\sim \|\Delta z\|_{H^{2n}}
+\|\Delta \rho\|_{L^2}$  , so 
\begin{equation}\label{equivnorm}
\widetilde{\|}\Delta V\|_{\mathcal{H}^{2n}}\sim \|\Delta V\|_{\mathcal{H}^{2n}}.
\end{equation}
The main result is the following:
\begin{prop}\label{mainenergy}
Let $V_{\text{ref}}^a$ be a reference state smooth, bounded with its derivatives rapidly decaying at infinity. 
Let $V^a=(\rho^a,v^a)=V_{\text{ref}}^a+U^a$ be an
approximate solution of  \eqref{EK}
\begin{equation*}
 \partial_tV^a=J\partial_x\delta H[V^a]+f^a,
\end{equation*}
and
$V$ a solution of \eqref{EK} such that $U=V-V_{\text{ref}}^a\in \mathcal{H}^{2n}, \ n\geq 1$. 
Then the estimate holds 
\begin{equation*}
\bigg|\frac{1}{2}\frac{d}{dt}\widetilde{\|}\Delta V\|_{\mathcal{H}^{2n}}\bigg|\leq
C\bigg(\|U\|_{\mathcal{H}^{2n}}+\|U^a\|_{\mathcal{H}^{2n+2}}
+\big\|1/\rho+1/\rho^a\bigg\|_{L^\infty}\bigg)\big(
\widetilde{\|}\Delta V\|_{\mathcal{H}^{2n}}+\|f^a\|_{\mathcal{H}^{2n}}\bigg),
\end{equation*}
with $C$ a continuous, positive nondecreasing function $\R^+\to \R^{+*}$.
\end{prop}

\begin{proof}
We recall the convention of section \ref{not}; the hidden constants in $\lesssim,\sim$ are as the function $C$
of the statement.
If $f^a=(f_1^a,f_2^a)$, the equations on $\Delta \rho,\Delta z$ are 
\begin{equation*}
\left\{
\begin{array}{lll}
\partial_t\Delta \rho+ \partial_x(\Delta \rho v+\rho^a\Delta v)&=&f_1^a,\\
\partial_t\Delta z+v\partial_x \Delta z+\Delta v\partial_x z^a+
iw\partial_x\Delta z+i\Delta w\partial_x z^a&&\\
+\partial_x\Delta g
+i\partial_x\big(a\partial_x\Delta z+\Delta a \partial_x z^a)&=&i\sqrt{\frac{K}{\rho^a}}\partial_xf_1^a+f_2^a:=h^a.
\end{array}\right.
\end{equation*}
$\|\Delta \rho\|_{L^2}$ is estimated by multiplying the first equation by 
$\Delta \rho$ and space integration
\begin{eqnarray}\label{controlrho}
\nonumber 
\bigg|\frac{1}{2}\frac{d}{dt}\|\Delta \rho\|_{L^2}^2\bigg|&\leq&
\|\partial_x \Delta \rho\|_{L^2}
(\|\Delta \rho\|_{L^2}\|v\|_{L^\infty}+\|\rho^a\|_{L^\infty}\|\Delta v\|_{L^2})
+\|\Delta \rho\|_{L^2}\|f_1^a\|_{L^2}\\
&\lesssim& \big(\|\Delta V\|_{\mathcal{H}^{2n}}+\|f^a\|_{\mathcal{H}^{2n}}^2\big)
\|\Delta V\|_{\mathcal{H}^{2n}}.
\end{eqnarray}
The main issue is thus to control $\Delta z$. Let us first note that 
\begin{equation}\label{estimforcage}
\|h^a\|_{H^{2n}}^2\lesssim \|\partial_xf_1^a\|_{H^{2n}}^2+\|f^a_2\|_{H^{2n}}^2\leq 
\|f^a\|_{\mathcal{H}^{2n}}^2.
\end{equation}
For $0\leq k\leq n$, we apply $a^k\sqrt{\rho}\partial_x^{2k}:=\varphi_k
\partial_x^{2k}$ to the second 
equation. Denoting $\Delta z_k=\varphi_k\partial_x^{2k}\Delta z$ we find after some commutations 
\begin{eqnarray}\nonumber
 \partial_t\Delta z_k +v\partial_x\Delta z_k+i\partial_x(a\partial_x \Delta z_k)
 +i(\varphi_kw+2k\partial_x(a)\varphi_k-2a\partial_x\varphi_k)\partial_x^{2k+1}\Delta z&&\\
\label{eqzn} +i\varphi_k\partial_x^{2k+1}(\Delta a\partial_x z^a)&=&
 R+\varphi_k\partial_x^{2k}h^a\hspace{5mm}
\end{eqnarray}
where $R$ is a remainder term containing derivatives of $\Delta z$
of order at most $2k$, and derivatives of $z^a$ of order at most 
$2k+2$
\begin{eqnarray*}
R&=&[z\partial_x,\varphi_k\partial_x^{2k}]\Delta z-i\varphi_k\partial_x^{2k}
(\Delta v\partial_xz^a+i\Delta w\partial_xz^a)\\
&&
+i[\partial_x(a\partial_x\cdot),\varphi_k\partial_x^{2k}]\Delta z
+2ik\partial_x(a)\varphi_k\partial_x^{2k+1}\Delta z\\
&&-\varphi_k\partial_x^{2k+1}(\Delta g)
+i\varphi_k\partial_x^{2k+1}(\Delta a\partial_xz^a)
-\varphi_k'\partial_x(\rho v)\partial_x^{2k}\Delta z.
\end{eqnarray*}
By construction, 
\begin{eqnarray*}
\varphi_kw+2k\partial_x(a)\varphi_k-2a\partial_x\varphi_k&=&
a^k\sqrt{K}+2k\sqrt{\rho K}\\
&=&\bigg(\sqrt{\frac{K}{\rho}}a^k\sqrt{\rho}+2ka^ka'\sqrt{\rho}
-2ka^ka'\sqrt{\rho}-\frac{a^{k+1}}{\sqrt{\rho}}\bigg)\partial_x \rho\\
&=& 0.
\end{eqnarray*}
Therefore, multiplying $\eqref{eqzn}$ by $\Delta z_k$ and integrating, 
\begin{eqnarray}\label{controlz}
\bigg|\frac{d}{dt}\|\Delta z_k\|_{L^2}^2\bigg|&\lesssim& 
(\|v\|_{L^\infty}\|\Delta z_k\|_{L^2}
+\|R\|_{L^2}+C\|z^a\|_{H^{2k+2}}\|\Delta z\|_{H^{2k}}+\|\varphi_k\partial_x^{2k}h^a\|_{L^2})\|\Delta z_k\|_{L^2}.
\end{eqnarray}
Using section \ref{calcsobo} and Faa di Bruno formula $\|R\|_{L^2}\lesssim
\|\Delta V\|_{\mathcal{H}^{2k}}$, moreover from \eqref{estimforcage}
$\|\varphi_k\partial_x^{2k}h^a\|_{L^2}\lesssim \|f^a\|_{\mathcal{H}^{2k}}$,
\eqref{controlz} rewrites
\begin{equation*}
\bigg|\frac{d}{dt}\|\Delta z_k\|_{L^2}^2\bigg|\lesssim \|\Delta V\|_{\mathcal{H}^{2k}}\big(\|\Delta V\|_{\mathcal{H}^{2k}}+
+\|f^a\|_{\mathcal{H}^{2k}}\big).
\end{equation*}
Thanks to \eqref{equivnorm}, $\|\Delta V\|_{\mathcal{H}^{2k}}
\lesssim \widetilde{\|}\Delta V\|_{\mathcal{H}^{2n}}$. Adding estimates 
\eqref{controlrho} and \eqref{controlz} with $k=0$ and $k=n$ we 
conclude 
\begin{equation*}
\bigg|\frac{d}{dt}\widetilde{\|}\Delta V\|_{\mathcal{H}^{2n}}^2
\bigg|=
\bigg|\frac{d}{dt}\big( \|\Delta \rho\|_{L^2}^2
+\|\sqrt{\rho}\Delta z\|_{L^2}^2
+\|\varphi_n\partial_x^{2n}\Delta z\|_{L^2}^2)\bigg|
\lesssim \widetilde{\|}\Delta V\|_{\mathcal{H}^{2n}}\big(
\widetilde{\|}\Delta V\|_{\mathcal{H}^{2n}}+\|f^a\|_{\mathcal{H}^{2n}}\big).
\end{equation*}

\end{proof}

\section{Linear estimates}\label{seclin}
This section is devoted to estimates in $\mathcal{H}^n$ 
on the flows associated to $J\partial_x\delta^2H[V^c]$ ($V^c$
a traveling wave) and 
$J\partial_x\delta^2H[S]$.
\\
We recall the notation 
$\di \delta H[V]=
\begin{pmatrix}
-K\partial_x^2\rho-\frac{1}{2}K'(\partial_x\rho)^2+g(\rho)+v^2/2\\
\rho v
\end{pmatrix}
$, in the same spirit 
\begin{equation}\label{d2H}
 \delta^2H[V]
 \begin{pmatrix}
  r\\
  u
 \end{pmatrix}
= \begin{pmatrix}
   (-K'(\rho)\partial_x^2\rho-\frac{1}{2}K''(\rho)(\partial_x\rho)^2+g'(\rho)\big)r
   -\partial_x(K\partial_xr) +uv
   \\
   \rho u+rv 
  \end{pmatrix},
\end{equation}
or in a matrix operator notation 
\begin{equation*}
 \delta^2H[V]=
 \begin{pmatrix}
   \big(-K'(\rho)\partial_x^2\rho-\frac{1}{2}K''(\rho)(\partial_x\rho)^2+g'(\rho)\big)
   -\partial_x(K\partial_x\cdot) & v
   \\
   v & \rho 
 \end{pmatrix}.
\end{equation*}
As can be expected, $\delta^2H$ is a symmetric operator.
Recalling $\langle\cdot,\cdot\rangle$ is the $L^2$ scalar product, 
we shall use frequently that 
\begin{eqnarray}\nonumber
\langle  \delta^2H[V]
 \begin{pmatrix}
  r\\
  u
 \end{pmatrix},
 \begin{pmatrix}
  r_1\\
  u_1
 \end{pmatrix}
 \rangle &=& \int_{\R} 
 \big(-K'(\rho)\partial_x^2\rho-\frac{1}{2}K''(\rho)(\partial_x\rho)^2+g'(\rho)\big) rr_1\\
\label{bilin} &&\hspace{1cm}+K\partial_xr\partial_xr_1+vur_1+vru_1+\rho uu_1dx,
\end{eqnarray}
so that $\delta^2H$ induces a continuous bilinear form on $\mathcal{H}^0$ if 
$V$ is smooth enough.

\subsection{Linear stability of a traveling wave}
\paragraph{The case of a kink} 
Let $V^c$ be a kink of speed $c$. 
The system \eqref{EK} linearized near $V^c$ reads after the change of variables $x\to x-ct$
\begin{equation*}
 \partial_tU(x,t)=J\partial_x\delta^2(H-cP)[V^c(x)]U(x,t).
\end{equation*}
We define a modified energy functional $E=H-cP$.
According to lemma $3$ in \cite{BDD2} (see also 
remark $2$ in this reference) kinks are always stable in 
the following sense:
\begin{lemma}
For any $U\in \mathcal{H}$ there exists a unique orthogonal decomposition
\begin{equation}\label{decompkink}
U=\alpha \partial_xV^c+W,\ \partial_x V^c\in \text{Ker}(\delta^2E)
\text{ and } \langle \delta^2E[V^c] U,U\rangle \gtrsim\|W\|_{\mathcal{H}^0}^2.
\end{equation}
\end{lemma}
For the link between linear stability and $\delta^2E$ being definite positive, see e.g.
theorem $3.1$ of Pego-Weinstein \cite{PegoWeinstein}.
\paragraph{The case of a soliton}
We consider a branch of solitons $V^c$. As it is more convenient 
here to work on $U^c$, we denote 
$P[U^c]=\int r^cu^c$ and abusively $\delta H[U^c]=\delta H[V^c]$. 
We recall (see \eqref{EKHam}) that $J=\begin{pmatrix}0 & -1 \\ -1 & 0\end{pmatrix}$ so that $\delta^2 P=-J$. 
From $-c\partial_xU^c=J\partial_x\delta H[U^c]$ we have a number of useful identities 
\begin{eqnarray}
 \forall\,U,V,\ \delta P[U]&=&\delta^2P[V]U=-JU,\\
\label{hamsol}
 (\delta H-c\delta P)[U^c]&:=& \delta E[U^c]\text{ is constant},
\\
\delta^2E[U^c]\cdot \partial_xU^c&=&0\text{ (differentiation of
\eqref{hamsol} in $x$)},\label{kernel}\\
\delta^2E[U^c]\partial_cU^c-\delta P[U^c]&=&0 \text{ (differentiation in $c$)} \label{jordan}\\
\Leftrightarrow
\delta^2E[U^c]\partial_cU^c&=&-JU^c.\label{cdiff}
 \end{eqnarray}

\subparagraph{Stability assumption} We assume that $U_c$ is stable, namely it satisfies :
\begin{equation*}
\frac{dP[U^s]}{ds}|_{s=c}<0.\label{stabass}
\end{equation*}
(see the appendix for a link with the so called Boussinesq
momentum of instability). This also implies that $\partial_cU^c$ is an unstable direction in the 
sense that 
\begin{equation*}
\langle \delta^2E[U^c]\partial_cU^c,\partial_cU^c\rangle = \langle \delta P[U^c],\partial_cU^c\rangle
=\frac{d}{dc}P[U^c]<0.
\end{equation*}
Let us first recall a result from \cite{BDD2} (proved for the formulation of the 
Euler-Korteweg system  in Lagrangian coordinates, see also \cite{Audiard8} appendix B for 
a proof in Eulerian coordinates).
\begin{lemma}\label{posE}
Under the stability assumption, the operator $\delta^2E[U^c]$ is block diagonal on the orthogonal decomposition 
$\mathcal{H}=\text{vect}(U_-)\oplus_\perp \text{vect}(\partial_xU^c)
\oplus_\perp\mathcal{G}$,
where $\partial_xU^c$ spans the kernel of $\delta^2E$, $U_-$ is a normalized eigenvector associated to the 
unique negative eigenvalue, and 
\begin{equation*}
\forall\,W\in \mathcal{G},\ \langle \delta^2 E[V^c] W,W\rangle
\gtrsim \|W\|_{\mathcal{H}^0}^2.
\end{equation*}
\end{lemma}

\begin{lemma}
For $U\in \mathcal{H}$, there exists a unique orthogonal decomposition 
\begin{equation}\label{pseudodiag}
U=\alpha \delta P[U^c]+\beta \partial_x U^c+W,\ W\in (\delta P[U^c],\ \partial_x U^c)^{\perp} \text{ and 
}\langle \delta^2EU,U\rangle\gtrsim \|W\|_{\mathcal{H}^0}^2-C\alpha^2.
\end{equation}
\end{lemma}
\begin{rmq}
To underline the unity between this decomposition 
and \eqref{decompkink} in the case of a kink, let us point 
out that since the reference state is constant
\begin{equation*}
 \partial_xV^c=\partial_xU^c.
\end{equation*} 
\end{rmq}

\begin{proof}
The momentum being invariant by translation, $\langle \delta P[U^c],\partial_xU^c\rangle=0$ 
and according to \eqref{kernel}, $\partial_xU^c\in \text{Ker}(\delta^2E)$.
Therefore the only thing to prove is $\langle 
\delta^2EW,W\rangle\gtrsim \|W\|_{\mathcal{H}^0}^2$.
By contradiction we assume the existence of  $W\in (\delta P[U^c],\partial_x U^c)^\perp
\setminus\{0\}$ such that
$\langle \delta^2EW,W\rangle\leq 0$, then for any $(\alpha,\beta,\gamma)\in \R^3$, using identities \eqref{kernel},
\eqref{jordan}
\begin{eqnarray*}
\langle \delta^2E (\alpha \partial_c U^c+\beta\partial_x U^c+\gamma W),\ \alpha \partial_c U_c&+&\beta\partial_x U^c
+\gamma W\rangle \\
&=&\langle \delta^2E (\alpha \partial_c U^c+\gamma W),\ \alpha \partial_c U^c
+\gamma W\rangle \\
&=&\alpha^2\langle \delta P[U^c],\partial_cU^c\rangle +\gamma^2\langle \delta^2E W,W\rangle \\
&&+2\alpha \gamma
\langle \delta P[U^c],W\rangle\\
&=& \alpha^2\langle \delta P[U^c],\partial_cU^c\rangle +\gamma^2\langle \delta^2E W,W\rangle,
\end{eqnarray*}
by orthogonality. By definition,  $\langle \partial_xU^c,
W\rangle =0$ and $\langle \partial_cU^c,\delta P[U^c]\rangle <0$ therefore $(\partial_c U^c,\partial_xU^c,W)$ is free. But $\delta^2E$ is thus nonpositive 
on a dimension $3$ space, which contradicts lemma \ref{posE}. As a consequence 
\begin{equation}\label{weakpos}
\forall\, W\in (\delta P[U^c],\ \partial_x U^c)^{\perp}\setminus\{0\} ,\ 
\langle \delta^2EW,W\rangle >0.
\end{equation}
The improved inequality $\langle \delta^2EW,W\rangle \gtrsim \|W\|_{\mathcal{H}^0}^2$ follows from a (probably standard) compactness 
argument : consider a sequence $V_n$ of $(\delta P[U^c],\ \partial_x U^c)^{\perp}$ such that $\|V_n\|_{\mathcal{H}^0}=1$ 
and $\langle \delta^2EV_n,V_n\rangle\to 0$. Using lemma $\ref{posE}$ we write $V_n=\alpha_nU_-+\beta_n\partial_xU^c+
W_n$, $W_n\in \mathcal{G}$. 
By assumption, $\beta_n=0$ and up to an extraction, $\alpha_n\to_{n\to \infty} \alpha,\ W_n\rightharpoonup 
W \in \mathcal{G}$. Denoting $-\lambda_-$ the negative eigenvalue,
\begin{eqnarray*}
1&=&\|V_n\|_{\mH}^2=\alpha_n^2+\|W_n\|_{\mathcal{H}^0}^2=\alpha^2+\lim_n \|W_n\|_{\mH}^2,\\ 
0&=&\lim_{n\to \infty} \langle \delta^2EV_n,V_n\rangle =\lim_{n\to \infty}-\lambda_-\alpha_n^2
+\langle\delta^2EW_n,W_n\rangle\geq -\lambda_-\alpha^2+c\liminf_n \|W_n\|_{\mathcal{H}^0}^2.
\end{eqnarray*}
This implies $\alpha\neq 0$.
Let $V=\alpha U_-+W$, then by weak convergence
\begin{equation*}
\langle \delta^2EV,V\rangle =-\lambda_-\alpha^2+\langle \delta^2EW,W\rangle 
\leq -\lambda_-\alpha^2+\lim_n \langle \delta^2EW_n,W_n\rangle =0,
\end{equation*}
but since $V$ is the weak limit of $V_n$, it belongs to $(\delta P[U^c],\partial_xU^c)^\perp$, and \eqref{weakpos} implies $V=0$, 
which contradicts $\alpha\neq 0$.

\end{proof}

As a consequence, we deduce the following linear stability result, whose proof 
will be a guideline for the computations in the multi-soliton case.
\begin{theo}\label{stabmono}
 Under the stability assumption \eqref{stabass}, the solution of 
 \begin{equation*}
 \left\{
 \begin{array}{lll}
 \partial_tU(x,t)&=&J\partial_x\delta^2H[U^c(x-ct)]U(x,t),\\
 U|_{t=0}&=&U_0,
 \end{array}\right.
 \end{equation*} 
 satisfies for $t\in \R$
 \begin{equation*}
 \|U(t)\|_{\mathcal{H}^0}\lesssim (1+|t|) \|U_0\|_{\mathcal{H}^0}.
 \end{equation*}
\end{theo}
\begin{proof}
For conciseness we write $\delta^2H$ for $\delta^2H[U^c]$. Using 
$\delta^2P=-J$
\begin{eqnarray*}
\frac{d}{dt}\langle (\delta^2H-c\delta^2P)U,U\rangle &=&
\langle [\partial_t,\delta^2H]U,U\rangle +2\langle \delta^2HJ\partial_x\delta^2HU,U
\rangle +c\langle [\partial_x,\delta^2H]U,U\rangle \\
&=& \langle [\partial_t+c\partial_x,\delta^2H]U,U\rangle.
\end{eqnarray*}
Since the coefficients of the operator $\delta^2H$ only depend on $x-ct$, 
$[\partial_t+c\partial_x,\delta^2H]=0$, so
\begin{equation}\label{enlin}
\frac{d}{dt}\langle 
(\delta^2H-c\delta^2P) U,U\rangle=0.
\end{equation}
We use the decomposition \eqref{pseudodiag} for the solution 
$U(t)=\alpha(t) \delta P[U^c(x-ct)]+\beta(t) \partial_x U^c+W(t)$.
Since 
$\partial_x U^c \in \text{Ker}(\delta^2E)$ 
\begin{equation*}
 \alpha'(t)=\frac{\langle J\partial_x\delta^2HU,\delta^2PU^c\rangle+\langle 
 U,-c\partial_x\delta^2PU^c\rangle }{\langle \delta P[U_c], \delta P[u_c]\rangle }
 =\frac{\langle U,(\delta^2H-c\delta^2P)\partial_xU^c\rangle }
 {\langle \delta P[U_c], \delta P[u_c]\rangle }=0.
\end{equation*}
By the conservation \eqref{enlin} and the continuity of $\delta^2H$ as a bilinear form 
\eqref{bilin}
\begin{eqnarray*}
\langle \delta^2EU(0),U(0)\rangle=\langle \delta^2EU(t),U(t)\rangle =\langle \delta^2E(\alpha \delta P[U^c]+W), 
\alpha \delta P[U^c]+W\rangle
\gtrsim \|W\|_{\mathcal{H}^0}^2-C\alpha^2\\
\Rightarrow  \|W(t)\|_{\mathcal{H}^0}^2\lesssim \alpha(0)^2+\|U(0)\|_{\mathcal{H}^0}^2.
\end{eqnarray*}
Moreover $|\alpha(0)|=\langle U(0),\delta P[U^c]\rangle /\|\delta P[U^c]\|_{\mH}^2\rangle \lesssim \|U(0)\|_{\mathcal{H}^0}$.
The last term is estimated thanks to the bounds on $\alpha,V$ :
\begin{equation*}
| \beta'(t)|=\bigg|\frac{d}{dt}\frac{\langle U(t),\partial_xU^c\rangle}
{\|\partial_xU^c\|_{L^2}^2}\bigg|=
\bigg|\frac{\langle J\partial_x\delta^2E(\alpha \delta P[U^c]+V),\partial_xU^c\rangle}{\|\partial_xU^c\|_{L^2}^2}\bigg|
\lesssim \|U(0)\|_{L^2},
 \end{equation*}
and by integration, $|\beta(t)|\leq |\beta(0)|+|t|\|U(0)\|_{\mathcal{H}^0}$.
\end{proof}
\begin{rmq}
 The linear growth in time is unavoidable, indeed we have $$J\partial_x\delta^2E\partial_c U^c
 =\partial_xJ\delta P[U^c]=\partial_xU^c\in \text{Ker}(J\partial_x\delta^2E),$$ therefore $(\partial_cU^c,\partial_xU^c)$
 is  associated to a Jordan block of the eigenvalue $0$ of $J\partial_x\delta^2E$.
\end{rmq}
\begin{rmq}
Of course to study the stability of a single soliton it is much more natural 
to do the galilean change of variable $y=x-ct$ and consider the \emph{autonomous} linear 
problem $\partial_tU=(\delta^2H-c\delta^2P)[U^c(y)]U$. The proof of theorem 
\ref{stabmono} is simplified in this frame. Nevertheless, when considering 
multi-soliton such a change variable is not available and this first simple case 
is a good warm up before the more technical computations of section \ref{stabdur}.
\end{rmq}
\subsection{Stability near multiple traveling waves}\label{stabdur}
We recall that
the multi-soliton is defined as 
$$S(x,t)=V^{c_1}(x-c_1t)+\sum_{k=2}^nU^{c_k}(x-c_kt-\sum_{j=2}^kA_j)=V_1+\sum_2^nU_k,\ A_j\geq A,$$ where $V^{c_k}=V^{c_k}_{\text{ref}}+U^{c_k}$ are traveling waves, $c_1<c_2\cdots<c_n$.
We have exponential decay
\begin{equation}\label{decay}
\forall\,p\geq 0,\ 1\leq k\leq n,\ 
\exists\,\alpha>0:\ |\partial_x^pU_k(x,t)|\lesssim e^{-\alpha|x-c_kt-\sum_2^kA_j|}.
\end{equation}
The aim of this section is to get bounds on the flow 
associated to $J\partial_x\delta^2H[S+\eta]$ where $\eta$ is a small perturbation of limited 
smoothness that 
depends on $x$ and $t$. 
When there is no ambiguity, we write $\delta^2H$ for $\delta^2H[S+\eta]$ and 
\begin{equation*}
\delta^2E_k:=\delta^2H[V_k]-c_k\delta^2P.
\end{equation*}

\begin{lemma}\label{estim0}
For $s\geq 0$, let $U$ solve
\begin{equation}\label{linmultisoliton}
\left\{
\begin{array}{ll}
\partial_tU=J\partial_x\delta^2H[S+\eta]U,\\
U|_{t=s}=U_0,
\end{array}\right.
\end{equation}
with $\eta$ a smooth perturbation. 
There exist $C$ and $\varepsilon_0$ such that for 
$\varepsilon:=1/A+\|\eta\|_{X^1}\leq \varepsilon_0$, 
\begin{equation*}
\forall\, t\geq 0,\ \|U(t)\|_{\mathcal{H}^0}\leq C(1+|t-s|)e^{C\varepsilon^{1/4}|t-s|}
\|U_0\|_{\mathcal{H}^0}.
\end{equation*}
\end{lemma}
\begin{rmq}
The estimate is not true for $t\leq 0$ as the key argument is that the distance 
between the traveling waves must be (in some sense) larger than $A$.
\end{rmq}

The proof requires some preliminaries that will be used through the section. 
Let $2c_0=\inf_{j<k} c_k-c_j$, and for $1\leq k<n$, $c_{k+1/2}=(c_k+c_{k+1})/2$.
We first define localizing functions : pick a nondecreasing $\chi\in C^\infty(\R)$, 
$\text{supp}(\chi)= [0,\infty],\ \chi|_{[1/2,\infty)}=1, 0<\chi<1$ on $(0,1/2)$, and set
\begin{eqnarray*}
\varphi_1(x,t)&=&1-\chi\bigg(\frac{x-c_{1+1/2}t-A_2/2}{A_2}\bigg),\\
\forall\, 2\leq k<n,\ \varphi_k(x,t)&=&\chi\bigg(\frac{x-c_{k-1/2}t-\big(
\sum_{2}^{k}A_j-A_k/2\big)}{A_k}\bigg)
\\&&\hspace{2cm}-\chi\bigg(\frac{x-c_{k+1/2}t-\big(\sum_2^kA_j+A_{k+1}/2\big)}{A_{k+1}}\bigg),\\
\varphi_n(x,t)&=&\chi\bigg(\frac{x-c_{n-1/2}t-\big(\sum_2^{n}A_j-A_{n}/2\big)}
{A_n}\bigg).
\end{eqnarray*}
It is easily seen that 
\begin{eqnarray*}
\text{supp}(\varphi_1)&=& (-\infty, c_{3/2}t+A_2],\\
\forall\,2\leq k\leq n-1,\  \text{supp}(\varphi_k)&=&[c_{k-1/2}t
+\sum_2^{k}A_j-\frac{A_k}{2},c_{k+1/2}t+\sum_2^{k+1}A_j],\\
\text{supp}(\varphi_n)&=&[c_{n-1/2}t+\sum_2^{n}A_j-A_n/2,\infty),
\end{eqnarray*}
and  $\di \sum_{k=1}^n \varphi_k^2\geq c>0$ for some constant independent of $x$. The localizing functions are then defined as 
\begin{equation}\label{partition}
 \chi_j=\frac{\varphi_j}{\sqrt{\sum_1^n \varphi_j^2}}\text{ so that } \sum \chi_j^2=1.
\end{equation}
Note that $\varphi_j$ and $\chi_j$ have same support. Thanks to \eqref{decay} we have the following
estimates, uniformly for $A$ large
\begin{eqnarray}
\label{bornechiprime}
\|\partial_x^j\partial_t^k\chi_j\|_{L^\infty_{x,t}}&=&O(1/A^{k+j})\text{ (slow variation )},\\
\label{bornechiu}
\ \forall\,j\neq k,\ (p,q)\in \N^2,\ r\geq 1,\ \exists\,\alpha>0:
\|\partial_x^{p}\partial_t^{q}U_k\|_{L^r_x(\text{supp}(\chi_j))}
&=&O(e^{-\alpha c_0t}/A), \\
\text{if } (p,q)\neq (0,0),\ r\geq 1,\
\|\partial_x^{p}\partial_t^{q}V_k\|_{L^r_x(\text{supp}(\chi_j))}
&=&O(e^{-\alpha c_0t}/A),\\
\nonumber&&\text{ (support decorrelation)}.
\end{eqnarray}

\begin{proof}[Proof of lemma \ref{estim0}]
In the spirit of the proof of theorem \ref{stabmono} we define the modified energy 
\begin{equation}
\widetilde{E}(t)=\langle \delta^2H[S+\eta]U(t),U(t)\rangle 
-\sum_{k=1}^n c_k\langle \delta^2P\chi_k U(t),\chi_k U(t)\rangle.
\end{equation}
Similarly to theorem \ref{stabmono}, the proof has three steps :
1) Control of $d\widetilde{E}/dt$, 2) control of $\|U(t)\|_{\mathcal{H}^0}^2$ by 
$\widetilde{E}$ up to a finite number of parameters, 3) Control of these parameters.\\
\vspace{2mm}\\
\textit{Step $1$:   Control of $d\widetilde{E}/dt$.} 
From basic computations, using $\delta^2P=-J,J^2=I$,
\begin{eqnarray*}
\frac{d}{dt} \widetilde{E}&=&\langle ([\partial_t,\delta^2H]+\delta^2HJ\partial_x\delta^2H)U,U\rangle 
+\sum_{k=1}^n2c_k\langle  \chi_k' JU,\chi_kU\rangle\\
&&+\sum_{k=1}^nc_k\bigg(\langle \chi_k\partial_x\delta^2H U,\chi_kU\rangle 
+\big(\langle \chi_kU,\chi_k\partial_x\delta^2H U\rangle\bigg)\\
&=&\langle [\partial_t,\delta^2H]U,U\rangle 
+\sum_{k=1}^n2c_k\langle  \chi_k' JU,\chi_kU\rangle\\
&&+\sum_{k=1}^nc_k\bigg(\langle [\chi_k^2,\partial_x\delta^2H] U,U\rangle 
+\langle [\partial_x,\delta^2H]\chi_k^2U,U\rangle\bigg)\\
&=&\sum_{k=1}^n \langle ([\partial_t,\delta^2H]+c_k[\partial_x,\delta^2H])\chi_k^2 U,U\rangle 
+c_k\big(2\langle  \chi_k' JU,\chi_kU\rangle
+c_k\langle [\chi_k^2,\partial_x\delta^2H]U,U\rangle\big)\\
&=&\sum_{k=1}^n C_{1,k}(t)+C_{2,k}(t)+C_{3,k}(t).
\end{eqnarray*}
We first point out that $X^1$ controls the 
$L^\infty$ norm, therefore for $\|\eta\|_{X^1}$ small enough
the density of $S+\eta$ remains bounded away from $0$ and the 
computations rules in \eqref{calcsobo},\eqref{calcsobo2} can be 
applied.\\
$C_{2,k}$ and $C_{3,k}$ are not difficult to control : let us write $[\chi_k^2,\partial_x\delta^2H]=(L_{i,j})_{1\leq i,j\leq 2}$ as a matrix of 
operators, $S+\eta=\begin{pmatrix}\rho_1\\v_1\end{pmatrix}$ and detail the estimate for 
$\langle L_{1,1}r,r\rangle$
\begin{eqnarray}
\nonumber
\langle L_{1,1}r,r\rangle&=&\langle \big[\chi_k^2, \partial_x\big(
(g'-K''(\partial_x\rho_1)^2-K'\partial_x^2\rho_1)-\partial_xK\partial_x\big)\big]r,\ 
r\rangle \\
\label{detail}
&=&-\langle 2\chi_k\partial_x(\chi_k)(g'-K''(\partial_x\rho_1)^2-K'\partial_x^2\rho_1)r,r
\rangle \\
\label{detail2}
&&-2\langle \chi_k\partial_x\chi_k\partial_x(K\partial_xr),r\rangle -
\langle [\chi_k^2,\partial_xK\partial_x]r,\partial_xr\rangle .
\end{eqnarray}
Using the Sobolev estimates \eqref{calcsobo} and \eqref{bornechiprime}, we find 
$\eqref{detail}\lesssim\|r\|_{H^1}^2/A$, the second one is estimated 
by an integration by part
\begin{equation*}
\big|-2\langle \chi_k\partial_x\chi_k\partial_x(K\partial_xr),r\rangle\big|
\leq \int_{\R}\big|K\partial_xr\partial_x\big(2\chi_k\partial_x\chi_k\,r\big)\big|dx
\lesssim \frac{1}{A}\|r(t)\|_{H^1}^2.
\end{equation*}
The last term in \eqref{detail2} is estimated similarly with the explicit commutator 
formula $[\chi_k^2,\partial_xK\partial_x]r=-2\partial_x(\chi_k^2)K\partial_xr
-\partial_x(\chi_k^2)\partial_x(Kr)$. Similar computations eventually 
lead to 
\begin{equation}\label{controlC2C3}
|C_{2,k}+C_{3,k}|\lesssim \frac{1}{A}\|U\|_{\mathcal{H}^0}^2.
\end{equation}
To bound $C_{1,k}$ we introduce the (bilinear) operator $\delta^3H$ such that 
\begin{equation}\label{derH}
[\partial_t,\delta^2H[S+\eta]]U=\delta^3H[S+\eta](U,\partial_t(S+\eta)).
\end{equation}
This is merely a convenient notation, as writing 
$S+\eta=\begin{pmatrix}  
\rho_1\\ v_1
\end{pmatrix}
$,
$\delta^3H(\cdot,\partial_t(S+\eta))$ is explicitly 
\begin{eqnarray}
\label{structop}
\delta^3H[S+\eta](\cdot,\partial_t(S+\eta))&=&
\begin{pmatrix}
\mathcal{M}_t& \partial_tv_1\\
\partial_t v_1 & \partial_t\rho_1
\end{pmatrix},\\
\nonumber 
\text{with }\mathcal{M}_tr&=&\bigg(g''\partial_t\rho_1
-\frac{K'''\partial_t\rho_1(\partial_x\rho_1)^2
+2K''\partial_x\rho_1\partial_{xt}\rho_1}{2}\bigg)r\\
\nonumber
&&-(K''\partial_t\rho_1\partial_x^2\rho_1+K'\partial_{xxt}\rho_1)r
-\partial_x (K'\partial_t\rho_1\partial_xr),
\end{eqnarray}
and we use the same notation for $[\partial_x,\delta^2H]:=\delta^3H[S](\cdot,\partial_xS)$.
We can thus rewrite using $S=V_1+\sum_{j=2}^n U_j$, with 
$\partial_t V_1=-c_1\partial_xV_1$,
$\partial_tU_j=-c_j\partial_xU_j$, for $k>1$
\begin{eqnarray*}
 C_{1,k}(t)&=&\langle \delta^3H(U,\partial_tS+c_k\partial_xS)\chi_k^2U,U\rangle
 +\langle \delta^3H(U,\partial_t\eta+c_k\partial_x\eta)\chi_k^2U,U\rangle\\
 &=& \langle \delta^3H(U,-c_1\partial_xV_1
 -\sum_{j\neq k}c_j\partial_xU_j)\chi_k^2U,U\rangle
 +\langle \delta^3H(U,\partial_t\eta+c_k\partial_x\eta)\chi_k^2U,U\rangle,
\end{eqnarray*}
and if $k=1$
\begin{eqnarray*}
 C_{1,1}(t)&=&
 \langle \delta^3H(U,-\sum_{j=2}^nc_j\partial_xU_j)\chi_1^2U,U\rangle
 +\langle \delta^3H(U,\partial_t\eta+c_1\partial_x\eta)\chi_1^2U,U\rangle.
\end{eqnarray*}
Now using the support decorrelation property \eqref{bornechiu} and the explicit form \eqref{structop} 
of $\delta^3H$ 
we obtain in both cases
\begin{equation}
 \label{estimC1}
|C_{1,k}(t)|\lesssim  \frac{e^{-\alpha c_0 t}}{A}\|U(t)\|_{\mathcal{H}^0}^2+(\|\eta(t)\|_{\mathcal{H}^2}+\|\partial_t\eta(t)\|_{\mathcal{H}^1})
\|U(t)\|_{\mathcal{H}^0}^2.
\end{equation}
Adding this estimate with \eqref{controlC2C3} gives
\begin{equation}
 \label{derivenergie}
 \big|\widetilde{E}'(t)\big|\lesssim \varepsilon\|U(t)\|_{\mathcal{H}^0}^2.
\vspace{2mm}
 \end{equation}
\textit{Step $2$: Lower bounds for $\widetilde{E}$.} The key here is the decompositions \eqref{decompkink} and
\eqref{pseudodiag}. For  $1\leq k\leq n$ we set 
$\chi_kU(t)=\alpha_k(t)\delta P[U_k]+\beta_k(t)\partial_x V_k+W_k(t)$, with the convention that if $V_1$ is a kink, $\alpha_1=0$ (in this case the relevant decomposition is \eqref{decompkink}).
Using the translation invariance of the $L^2$ norm, the lower bound in \eqref{pseudodiag} gives for 
some $m,C>0$
$$\forall\,1\leq k\leq n,\ \langle \delta^2E[U_k]\chi_kU,\chi_kU\rangle 
\geq m\|W_k\|_{\mathcal{H}^0}^2-C\alpha_k^2.$$  
According to this, we split $\widetilde{E}$ as a sum of localized terms and remainders:
\begin{eqnarray*}
\nonumber
 \widetilde{E}&=&\langle \delta^2H[S+\eta]U,U\rangle -\sum_{k=1}^n c_k\langle \delta^2P\chi_kU,
 \chi_kU\rangle\\
 \nonumber
 &=&\sum_{k=1}^n\langle \delta^2H[S+\eta]\chi_k^2U,U\rangle - c_k\langle \delta^2P\chi_kU,
 \chi_kU\rangle,
 \end{eqnarray*}
 so commuting $\delta^2H$ and $\chi_k$ we obtain
 \begin{eqnarray}
 \nonumber \widetilde{E}(t)
 &=&\sum_{k=1}^n\langle \delta^2E_k\chi_kU,\chi_kU\rangle 
 +\langle (\delta^2H[S+\eta]-\delta^2H[U_k])\chi_kU,\chi_kU\rangle\\
\nonumber
 &&\hspace{1cm}+\langle [\delta^2H[S+\eta],\chi_k]\chi_kU,U\rangle\\
 \nonumber
 &\geq & \sum_{k=1}^n m\|W_k(t)\|_{\mathcal{H}^0}^2-C\alpha_k^2(t)
 +\langle (\delta^2H[S+\eta]-\delta^2H[V_k])\chi_kU,\chi_kU\rangle\\
 \label{minorEstep1}&&\hspace{1cm}+\langle [\delta^2H,\chi_k]\chi_kU,U\rangle.
 \end{eqnarray}
The last term is estimated as in \eqref{detail2},
\begin{equation}\label{resteE1}
\big|\langle [\delta^2H,\chi_k]\chi_kU(t),U(t)\rangle\big|
\lesssim \varepsilon\|U(t)\|_{\mathcal{H}^0}^2.
\end{equation}
Thanks to the support decorrelation \eqref{bornechiu} and calculus rules \eqref{calcsobo}, one can check 
\begin{equation}\label{resteE2}
\langle (\delta^2H[S+\eta]-\delta^2H[V_k])\chi_kU,\chi_kU\rangle\lesssim 
\bigg(\frac{e^{-\alpha c_0t}}{A}+\|\eta\|_{\mathcal{H}^2}\bigg)\|U(t)\|_{\mathcal{H}^0}^2.
\end{equation}
For example the term associated to $\partial_x(K\partial_xr)$ is controlled as follows
\begin{eqnarray*}
\big|\langle \partial_x\big((K(S+\eta)-K(V_k))\partial_x (\chi_kr)\big),\,\chi_kr\rangle\big|
&\leq& 
\big|\|K(S+ \eta)-K(V_k)\|_{L^\infty(\text{supp}(\chi_k))}\|\partial_x(\chi_kr)\|_{L^2}^2\\
&\lesssim & \bigg(\frac{e^{-\alpha c_0t}}{A}+\|\eta(t)\|_{\mathcal{H}^0}\bigg)
\|r(t)\|_{H^1}^2.
\end{eqnarray*}
Note that $\|U\|_{\mathcal{H}^0}^2\lesssim \sum_1^n\|W_k\|_{\mathcal{H}^0}^2+\alpha_k^2+\beta_k^2$,
so for $\varepsilon$ small enough, from \eqref{minorEstep1},\eqref{resteE1},\eqref{resteE2}, there exists constants
$m,C_0,C_1$ ($m$ is not the same as in \eqref{minorEstep1}) such that
\begin{equation}\label{minorEalter}
\widetilde{E}(t)\geq \sum_{k=1}^n m\|W_k(t)\|_{\mathcal{H}^0}^2-C_0\alpha_k^2(t)
-C_1\varepsilon\beta_k^2(t).
\end{equation}
\textit{Step $3$: Control of the parameters.} Once more it is a matter of repeating 
the proof of theorem \ref{stabmono} with some commutators. Let us start with 
$(\alpha_k(t))_{1\leq k\leq n}$ ($k>2$ when $V_1$ is a kink): 
\begin{eqnarray*}
\alpha_k'(t)&=&\frac{d}{dt}\frac{\langle \chi_kU,
\delta^2P U_k\rangle}{\|\delta^2P U_k\|_{L^2}^2}\\
&=&\frac{\langle (\partial_t\chi_k)U,\delta^2P U_k\rangle 
+\langle \chi_kJ\partial_x\delta^2H[S+\eta]U,\delta^2P U_k\rangle +
\langle \chi_kU,-c_k\delta^2P \partial_xU_k\rangle}{\|\delta^2P U_k\|_{L^2}^2}\\
&=&\frac{\langle (\partial_t\chi_k)U,\delta^2P U_k\rangle 
+\langle [\chi_k,J\partial_x\delta^2H[S+\eta]]U,\delta^2P U_k\rangle
}{\|\delta^2P U_k\|_{L^2}^2}
\\
&&+\frac{\langle J\partial_x(\delta^2H[S+\eta]-\delta^2H[V_k])
\chi_kU,\delta^2P U_k\rangle+\langle \chi_kU,\delta^2E_k\partial_xU_k\rangle}
{\|\delta^2P U_k\|_{L^2}^2}.
\end{eqnarray*}
There are four terms. The fourth one is actually $0$, thanks to identity 
\eqref{kernel}. From the same argument as for $C_{2,k},C_{3,k}$ in \eqref{detail2}, 
the first and second ones are $O(|U(t)\|_{\mathcal{H}^0}/A)$.
Using integration by part, the smoothness of $U_k$ and the
same argument as for \eqref{resteE2}, the third one is 
$O(\|U(t)\|_{\mathcal{H}^0}(e^{-\alpha c_0t}/A+\|\eta\|_{\mathcal{H}^2}))$. To summarize
\begin{equation}\label{estimalpha}
\forall\,1\leq k\leq n,\ t\geq 0,\ |\alpha_k'(t)|\lesssim \varepsilon\|U(t)\|_{\mathcal{H}^0}.
\end{equation}
We bound now $\beta_k(t)$:
\begin{eqnarray*}
 \|\partial_xV_k\|_{L^2}^2\beta_k'(t)&=&\frac{d}{dt}\langle \chi_kU,\partial_x V_k\rangle
 \\
 &=&\langle (\partial_t\chi_k)U,\partial_xV_k\rangle
 +\langle \chi_kJ\partial_x\delta^2HU,\partial_xV_k\rangle 
 +\langle \chi_kU,-c_k\partial_x^2V_k\rangle
\\
&=&\langle (\partial_t\chi_k)U,\partial_xU_k\rangle
 +\langle J\partial_x\delta^2E_k(\chi_kU),\partial_x V_k\rangle \\
&& +\langle \chi_kJ\partial_x(\delta^2H[S+\eta]-\delta^2H[V_k])U,\partial_xV_k\rangle
 +\langle [\chi_k,J\partial_x\delta^2H[V_k]]U,\partial_xV_k \rangle.
 \end{eqnarray*}
Since $\chi_kU=\alpha_k \delta P[U_k]+\beta_k\partial_xV_k+W_k$, with 
$\delta^2E_k\partial_xV_k=0$, we have 
\begin{equation*}
\big|\langle J\partial_x\delta^2E_k\big(\chi_kU(t)\big),\partial_x V_k(t)\rangle \big|
\lesssim |\alpha_k(t)|+\|W_k(t)\|_{\mathcal{H}^0}.
\end{equation*}
The other terms are estimated as for $\alpha_k'$, leading to 
\begin{equation}\label{estimbeta}
 |\beta_k'(t)|\lesssim |\alpha_k(t)|+\|W_k(t)\|_{\mathcal{H}^0}+
 \varepsilon\|U(t)\|_{\mathcal{H}^0}.
\end{equation}
\textit{Conclusion.}
Let us rewrite \eqref{derivenergie},\eqref{estimalpha},\eqref{estimbeta} : there exists some $C>0$ such that 
for $\varepsilon$ small enough 
\begin{eqnarray*}
\big| \widetilde{E}'(t)\big|&\leq& C\varepsilon\|U(t)\|_{\mathcal{H}^0}^2ds,
\\
|\alpha_k'(t)|&\leq& C\varepsilon\|U(t)\|_{\mathcal{H}^0},
\\
|\beta_k'(t)|&\leq& C\big(|\alpha_k(t)|
+\|V_k(t)\|_{\mathcal{H}^0}+\varepsilon\|U(t)\|_{\mathcal{H}^0}\big).
\end{eqnarray*}
With the same constants as in \eqref{minorEalter}, let 
$\di \widehat{E}(t):=\widetilde{E}(t)+\sum_1^n (C_0+m)\alpha_k^2+\varepsilon^{1/2}\beta_k^2$, 
then for $\varepsilon$ small enough $\widehat{E}(t)\gtrsim \sum_1^n \alpha_k^2+\|W_k\|_{\mathcal{H}^0}^2+\varepsilon^{1/2}\beta_k^2$, and
\begin{eqnarray*}
\big|\widehat{E}'(t)\big|&\leq& \sum_{k=1}^nC\varepsilon \big( \|W_k\|_{\mathcal{H}^0}^2
+\alpha_k^2+\beta_k^2\big)+C\varepsilon^{1/2}|\beta_k|\big(|\alpha_k|+\|W_k\|_{\mathcal{H}^0}
+\varepsilon\|U\|_{\mathcal{H}^0}\big)\\
&\leq& C\varepsilon^{1/2}\widetilde{E}+C\varepsilon^{1/2}\bigg(\frac{\alpha_k^2
+\|W_k\|_{\mathcal{H}^0}^2}{\varepsilon^{1/4}}+\varepsilon^{1/4}\beta_k^2\bigg)\\
&\leq& C\varepsilon^{1/4}\widehat{E}(t).
\end{eqnarray*}
With Gronwall's lemma and thanks to \eqref{minorEalter} we get 
\begin{equation*}
\sum_{k=1}^nm\big(\|W_k(t)\|_{\mathcal{H}^0}^2+\alpha_k^2(t)\big)
+\varepsilon^{1/2}\beta_k^2(t)\leq
\widehat{E}(t)\leq \widehat{E}(s)e^{C\varepsilon^{1/4}|t-s|}.
\end{equation*}
We can assume $\varepsilon\leq 1/(4C_1^2)$ so that $\varepsilon^{1/2}-C_1\varepsilon\geq \varepsilon^{1/2}/2$, and
\begin{equation*}
\sum_{k=1}^n\big(\|W_k(t)\|_{\mathcal{H}^0}^2+\alpha_k(t)^2\big)\lesssim
\|U(s)\|_{\mathcal{H}^0}^2e^{C\varepsilon^{1/4}|t-s|}
=\|U_0\|_{\mathcal{H}^0}^2e^{C\varepsilon^{1/4}|t-s|}.
\end{equation*}
To bound $\beta_k$ independently of $\varepsilon$, we plug the estimate above in the 
differential inequality \eqref{estimbeta}
\begin{eqnarray}
\nonumber\sum_{k=1}^n |\beta_k'(t)|&\leq& M\bigg(\sum_{k=1}^n|\alpha_k(t)|
+\|W_k(t)\|_{\mathcal{H}^0}+\varepsilon \|U(t)\|_{\mathcal{H}^0}\bigg)\\
&\leq& M_1\|U_0\|_{\mathcal{H}^0}e^{C\varepsilon^{1/4}|t-s|/2}
\label{diffbeta}
+\frac{C\varepsilon^{1/4}}{2}\sum_{k=1}^n |\beta_k(t)|.
\end{eqnarray}
(for $\varepsilon$ small enough so that $M\varepsilon\leq C\varepsilon^{1/4}/2$). 
\eqref{diffbeta} has the form $(e^{-\delta |t-s|}\varphi(t))'\leq
Me^{\mu|t-s|}$, so by integration on $[s,t]$
\begin{eqnarray*}
\sum_{k=1}^n |\beta_k(t)|&\leq& \sum_{k=1}^n|\beta_k(s)|e^{C\varepsilon^{1/4}|t-s|/2}+
\frac{2M_1\|U_0\|_{\mathcal{H}^0}}{C\varepsilon^{1/4}}\big(e^{C\varepsilon^{1/4}|t-s|}-1\big)\\
&\lesssim& \|U_0\|_{\mathcal{H}^0}\big(1+|t-s|\big)e^{C\varepsilon^{1/4}|t-s|}.
\end{eqnarray*}
Combining this with the estimate on $W_k,\alpha_k$, we conclude
\begin{equation*}
\|U(t)\|_{\mathcal{H}^0}\lesssim \|U_0\|_{\mathcal{H}^0}(1+|t-s|)e^{C\varepsilon^{1/4}|t-s|}.
\end{equation*}
\end{proof}
It is useful to restate the result of lemma \ref{estim0} in a slightly more abstract way: 
\begin{coro}\label{estimresolvent}
Let $R_\eta(t,s)$ the resolvent operator associated to $\partial_tU=J\partial_x\delta^2H[S+\eta]U$. \\
There exists $\varepsilon_0$ and $C$ such that if $1/A+\|\eta\|_{X^1}:=\varepsilon\leq \varepsilon_0,\ 
t, s\geq 0$:
\begin{equation*}
\|R_\eta(t,s)\|_{\mathcal{L}(\mathcal{H}^0)}\leq (1+|t-s|)Ce^{C\varepsilon^{1/4}|t-s|}.
\end{equation*}
\end{coro}

We deduce now similar estimates at any level of regularity.
\begin{theo}[Higher order estimates]\label{estimN}
Let $\varepsilon:=1/A+\|\eta\|_{X^{2n}}$, $n\in \N^*$. There exists 
$\varepsilon_n$, $C_{e,n}$, $C_n$ such that for any $\varepsilon\leq \varepsilon_n$,
we have the resolvent estimate
\begin{equation*}
\|R_\eta(t,s)\|_{\mathcal{L}(\mathcal{H}^{2n})}\leq (1+|t-s|)C_n
e^{C_{e,n}\varepsilon^{1/4}|t-s|} 
\end{equation*}
\end{theo}
\begin{proof}
The hamiltonian structure is useless here, so we 
denote for conciseness $L:=J\partial_x\delta^2H[S+\eta]$.
As for the energy estimate, an important issue is that $J\partial_x\delta^2H[S+\eta]$ does 
not commute with $\partial_t$, this will be overcome with the same method as 
for the zero order estimate.  The commutator  
$[\partial_t,L]=J\partial_x\delta^3H(\,\cdot\,,\partial_t(S+\eta)):=\delta L(\,\cdot\,,\partial_t(S+\eta))$ is not zero 
(for the definition of $\delta^3H$, see \eqref{derH} in the proof of lemma \ref{estim0}), 
thus to get higher order estimates it is more natural 
to use the operator $\di \widetilde{L}:=L+\sum_{k=1}^n c_k\chi_k^2\partial_x$.
Indeed using
$$[L,\widetilde{L}]=\sum_{k=1}^n[L,\ c_k\chi_k^2\partial_x]=
\sum_{k=1}^nc_k[L,\chi_k^2]\partial_x-c_k\chi_k^2\delta L(\cdot,\partial_x(S+\eta)),$$ we find for any $V$
\begin{eqnarray}\label{quasicommute}
 [\partial_t-L,\widetilde{L}]V =\sum_{k=1}^n\chi_k^2\delta L(V,(\partial_t+c_k\partial_x)(S+\eta))+
 2c_k\chi_k(\partial_t\chi_k)\partial_xV-[L,c_k\chi_k^2]\partial_xV,
\end{eqnarray}
therefore recalling $\partial_x V_j=\partial_xU_j$ for 
$j\geq 2$
\begin{eqnarray}
\partial_t\big(\widetilde{L}^nU\big)
&=&L\big(\widetilde{L}^nU\big)+\sum_{q=0}^{n-1}\widetilde{L}^q[\partial_t-L,\widetilde{L}]L^{n-q-1}U
\nonumber\\
&=&L(\widetilde{L}^nU)+\sum_{q=0}^{n-1}\widetilde{L}^q\bigg(\sum_{j\neq k}
\chi_k^2\delta L(\cdot,(c_k-c_j)\partial_x V_j)+\chi_k^2\delta L(\cdot,(\partial_t+c_k\partial_x)\eta)\bigg)
\widetilde{L}^{n-q-1}U\nonumber \\
\nonumber
&&+\sum_{q=0}^{n-1}\widetilde{L}^q\bigg(
\sum_{k=1}^n 2c_k\chi_k(\partial_t\chi_k)\partial_x-[L,c_k\chi_k^2]\partial_x\bigg)\widetilde{L}^{n-q-1}U
\\
\label{eqLU}
&=&L(\widetilde{L}^nU)+C_1+C_2.
\end{eqnarray}
Let us recall that if $S+\eta=\begin{pmatrix}\rho_1\\ v_1                              \end{pmatrix}$, for any $U=(r,u)^t$, we have 
\begin{equation*}
 LV=-\partial_x
 \begin{pmatrix}
u\rho_1+rv_1\\
 \bigg(g'-\frac{K''(\partial_x\rho_1)^2}{2}-K'\partial_x^2\rho_1\bigg)\,r -\partial_x(K\partial_x r) + 
uv_1
\end{pmatrix},
\end{equation*}
the first coefficient contains derivatives of $r,u,
\rho_1,v_1$ of order at most $1$, 
the second contains derivatives of $v_1,v$ of order at most $1$, and derivatives 
of $\rho_1,l$ of order at most $3$ so using the rules of section \ref{calcsobo},
\begin{eqnarray}\label{borneL}
\forall\,N\geq 0,\ 
\| LV\|_{\mathcal{H}^N}&\leq& C_N(\|\eta\|_{\mathcal{H}^{N+2}})
\|V\|_{\mathcal{H}^{N+2}},
\\
\Rightarrow \|\widetilde{L}V\|_{\mathcal{H}^N}&\leq& M_N(\|\eta\|_{\mathcal{H}^{N+2}})
\|V\|_{\mathcal{H}^{N+2}},
\end{eqnarray}
with $C_N,M_N$ positive locally bounded functions (we recall 
that $S$ is smooth and is unimportant in the estimates).
With this observation, estimate \eqref{bornechiprime} and computations similar 
to \eqref{detail2} we deduce 
\begin{equation*}
\|C_2(t)\|_{\mathcal{H}^0}
\lesssim  \frac{\|U(t)\|_{\mathcal{H}^{2n}}}{A}\leq \varepsilon\|U(t)\|_{\mathcal{H}^{2n}}.
\end{equation*}
Similarly, $C_1(t)$ is estimated thanks to \eqref{bornechiu} as for \eqref{estimC1},
\begin{eqnarray*}
\bigg\|C_1(t)\bigg\|_{\mathcal{H}^0}\lesssim 
\bigg(\frac{e^{-\alpha c_0t}}{A}+\|\eta\|_{X^{2n}}\bigg)
\|U(t)\|_{\mathcal{H}^{2n}}\leq \varepsilon\|U(t)\|_{\mathcal{H}^{2n}}.
\end{eqnarray*}
Conversely thanks to the interpolation estimate 
$\|\partial_x^p\varphi\|_{L^2}\leq \|\varphi\|_{L^2}^{1-p/q}\|\partial_x^q\varphi
\|_{L^2}^{p/q}$, and Young's inequality, 
$\|\partial_x^p\varphi\|_{L^2}\leq C(\varepsilon)\|\varphi\|_{L^2}
+\varepsilon\|\partial_x^q\varphi\|_{L^2}$ for any $\varepsilon>0$, $p<q$, therefore 
using once more the explicit formula for $\widetilde{L}V$, we obtain 
\begin{equation}\label{garding}
\|\widetilde{L}V\|_{\mathcal{H}^0}\gtrsim \|V\|_{\mathcal{H}^2}-C\|V\|_{\mathcal{H}^0},
\text{ and by induction }
\|\widetilde{L}^nU\|_{\mathcal{H}^0}\gtrsim \|U\|_{\mathcal{H}^{2n}}
-C_n\|U\|_{\mathcal{H}^0}.
\end{equation}
Equation \eqref{eqLU} is of the form $\partial_t V=LV+f$. We apply Duhamel's formula 
and corollary \ref{estimresolvent}
\begin{eqnarray}
\|\widetilde{L}^nU(t)\|_{\mathcal{H}^0}&\lesssim& (1+|t-s|)\|U_0\|_{\mathcal{H}^{2n}}
e^{C\varepsilon^{1/4} |t-s|}+\varepsilon
\int_s^t (1+|t-\tau|)e^{C\varepsilon^{1/4}|t-\tau|}\|U(\tau)\|_{\mathcal{H}^{2n}}ds\\
\label{duhamel1}
&\lesssim & (1+|t-s|)\|U_0\|_{\mathcal{H}^{2n}}e^{C\varepsilon^{1/4} |t-s|}+\varepsilon^{3/4}\int_s^t
e^{2C\varepsilon^{1/4}|t-\tau|}\|U(\tau)\|_{\mathcal{H}^{2n}}d\tau,
\end{eqnarray}
where the last estimate simply follows from $(1+s)e^s\lesssim e^{2s}$.\\
We use \eqref{duhamel1}, the bound $\|U(t)\|_{\mathcal{H}^0}\leq C(1+|t-s|)
e^{C\varepsilon^{1/4} |t-s|}\|U_0\|_{\mathcal{H}^0}$ of lemma \ref{estim0} and the lower bound \eqref{garding} with Gronwall's lemma
\begin{eqnarray*}
\|U(t)\|_{\mathcal{H}^{2n}}
&\lesssim& (1+|t-s|)e^{C\varepsilon^{1/4} |t-s|}\|U_0\|_{\mathcal{H}^{2n}}+
\varepsilon^{3/4}\int_s^te^{2C\varepsilon^{1/4}(t-\tau)}\|U(s)\|_{\mathcal{H}^{2n}}ds,\\
\Rightarrow \|U(t)\|_{\mathcal{H}^{2n}}&\lesssim&(1+|t-s|)\|U_0\|_{\mathcal{H}^{2n}}
e^{C(\varepsilon^{1/4}+\varepsilon^{3/4})|t-s|}.
\end{eqnarray*}
For $\varepsilon$ small $\varepsilon^{3/4}+\varepsilon^{1/4}=O(\varepsilon^{1/4})$ 
and the proof is complete.
\end{proof}

\section{Construction of an approximate solution}\label{secnewton}
We construct here an approximate solution $V^a$ close to the multi-soliton and 
such that the error 
$\partial_tV^a-J\partial_x\delta H[V^a]$ is rapidly decaying in $t$. It is done 
with Newton's algorithm, initialized with $S$ as the first approximate solution. 
\begin{theo}\label{solapp}
For any $n\in \N$, $\varepsilon>0, C_e>0$, there exists $A_0$ such that for 
$A\geq A_0$, there exists $U^a\in L^\infty_t\mathcal{H}^n$, 
$\alpha>0$ such that 
 \begin{equation*}
\forall\,t\geq 0,\ 
\|\partial_t(S+U^a)-J\partial_x\delta H[S+U^a]\|_{\mathcal{H}^n}\leq 
\varepsilon e^{-C_et}\text{ and }\|U^a(t)\|_{\mathcal{H}^n}\lesssim
\frac{e^{-\alpha c_0t}}{A}.
 \end{equation*}
\end{theo}
We recall Newton's algorithm: $S^0=S,\ f^0=\partial_tS^0-J\partial_x\delta H[S^0]$, $\eta^0=0$, and 
 recursively
 \begin{eqnarray*}
 \eta^{j+1}\text{ is the solution of }
 \left\{
 \begin{array}{ll}
 \partial_t\eta^{j+1}=J\partial_x\delta^2H[S^j]\eta^{j+1}
 -f^{j},\\ 
 \di \lim_{t\to\infty} \eta^{j+1}=0,
 \end{array}
 \right.\\
 S^{j+1}=S^j+\eta^{j+1},\ f^{j+1}=\partial_t S^{j+1}-J\partial_x\delta H[S^{j+1}].
 \end{eqnarray*}
Of course since $\partial_t S^{j+1} -J\partial_x \delta H[S^{j+1}]=
J\partial_x\delta H[S^j]+J\partial_x\delta^2H[S^j]\eta^{j+1}-J\partial_x\delta H[S^{j+1}]$
we need some Taylor expansion estimate :
\begin{lemma}\label{taylor}
For $V:=S+U$, $(U,\eta)\in (\mathcal{H}^{n+2})^2$ such that 
$\|V\|_{L^\infty}+\|\eta\|_{L^\infty}\leq \inf \rho_S/2$ (non vacuum 
condition), then
\begin{equation*}
\|J\partial_x\delta H[V+\eta]-J\partial_x\delta H[V]-J\partial_x\delta^2H[V]\eta\|
_{\mathcal{H}^n}\leq C(\|\eta\|_{\mathcal{H}^{n+2}}+
\|U\|_{\mathcal{H}^{n+2}})\|\eta\|_{\mathcal{H}^{n+2}}^2,
\end{equation*}
with $C$ continuous.
\end{lemma}
\begin{proof}
We set $V=\begin{pmatrix}\rho\\ v\end{pmatrix},\ 
U=\begin{pmatrix}r\\ u\end{pmatrix},\ \eta=\begin{pmatrix}\gamma\\ \omega\end{pmatrix}$. \\
Elementary  computations lead to 
\begin{eqnarray*}
J\partial_x\delta H[V+\eta]&-&J\partial_x\delta H[V]-J\partial_x\delta^2H[V]\eta\\
&=&
-\partial_x
\begin{pmatrix}
\gamma \omega \\
\di \frac{\omega^2}{2}+\big(g(\rho+\gamma)-g(\rho)-\gamma g'(\rho)\big)
-\big( K(\rho+\gamma)-K(\rho)\big)\partial_x^2\gamma\\
+\big(K(\rho)+\gamma K'(\rho)-K(\rho+\gamma)\big)\partial_x^2\rho
-\big( K'(\rho+\gamma)-K'(\rho)\big)\partial_x \rho\partial_x\gamma\\
-\frac{1}{2}K'(\rho+\gamma)(\partial_x \gamma)^2
\end{pmatrix}.
\end{eqnarray*}
We have using the rules in \ref{calcsobo} 
\begin{equation*}
 \|\partial_x(\gamma \omega)\|_{H^{n+1}}\lesssim \|\gamma\|_{H^{n+2}}\|\omega\|_{H^{n+2}}
 \leq \|\eta\|_{\mathcal{H}^{n+2}}^2.
\end{equation*}
For the second coordinate, we only treat a few terms. 
Thanks to the smallness assumption in $L^\infty$, 
$\inf\rho,\inf\rho+\gamma\geq \inf \rho_S/2$ so that we can 
use the composition rules again
\begin{eqnarray*}
\|\partial_x\big((K(\rho+\gamma)-K(\rho))\partial_x^2\gamma\big)\|_{H^n}&\lesssim& \|K(\rho+\gamma)-K(\rho)\|_{H^{n+1}}
\|\gamma\|_{H^{n+3}}\lesssim \|\gamma\|_{H^{n+3}}^2,\\
\|\partial_x\big(K(\rho)+\gamma K'(\rho)-K(\rho+\gamma))\partial_x^2\rho\big)\|_{H^n}&\lesssim& \|K(\rho)+\gamma K'(\rho)-K(\rho+\gamma)\|_{H^{n+1}}
\|r\|_{H^{n+3}}\\
&\lesssim& \|\gamma\|_{H^{n+2}}^2.
\end{eqnarray*}
After similar estimates for the other terms, we end up with 
\begin{equation*}
\|J\partial_x\delta H[U+\eta]-J\partial_x\delta H[U]-J\partial_x\delta^2H[U]\eta\|
_{\mathcal{H}^n}\lesssim \|\omega\|_{H^{n+2}}^2+\|\gamma\|_{H^{n+3}}^2=\|\eta\|_{\mathcal{H}^{n+2}}^2.
\end{equation*}

\end{proof}
\begin{proof}[Proof of theorem \ref{solapp}]
We fix some $k\in \N$ large, and iterate Newton's algorithm $k$ times.\\
The proof is divided in the following steps : 1. Control of $f^0$, 2. Control of 
$\eta^1$, 3. Iteration argument.\\
\textit{Step 1: Control of $f^0$.}
We use the partition of unity \eqref{partition}:
\begin{eqnarray*}
 \partial_tS-J\partial_x\delta H[S]=\sum_{j,k,\ j\neq k} \chi_j^2J\partial_x\delta H[V_k]
 +\sum_j\chi_j^2J\partial_x\big(\delta H[V_j]-\delta H\big[S\big]\big).
 \end{eqnarray*}
We set $V_j=\begin{pmatrix}\rho_j\\v_j\end{pmatrix}$, 
$S=\begin{pmatrix}\rho_S\\v_S\end{pmatrix}$, and explicit 
the second term
\begin{eqnarray*}
 \delta H[V_j]-\delta H[S]=
\begin{pmatrix}
\di -(\rho_S-\rho_j)v_j-\rho_S(v_S-v_j)\\
\di (u_S-u_j)(u_j+u_S)
+\big(K(\rho_j)-K(S)\big)\partial_x^2\rho_j\\
+K\big(S\big)\partial_x^2\big(\rho_j-S\big)\\
+\frac{1}{2}K'(\rho_j)(\partial_x \rho_j)^2-\frac{1}{2}K'(S)(\partial_x S)^2
\end{pmatrix}.
\end{eqnarray*}
Since all terms are smooth and spatially decorrelated, they
are small and exponentially decaying in $t$ for any 
Sobolev norms, for example
\begin{eqnarray*}
\bigg\|\chi_j\partial_x \big((K(\rho_j)-K(\rho_S))\partial_x^2\rho_j\big)\bigg\|_{H^{2n+2k}}
&\lesssim& \|\partial_x\rho_j\|_{H^{2n+2(k+1)}}\|K(\rho_j)-K(S)\|_{H^{2n+2k+1}(\text{supp}(\chi_j))}\\
&\lesssim& \frac{e^{-\alpha c_0t}}{A}.
\end{eqnarray*}
With similar computations, we find the existence of $C_0$ only depending on $S,n,k$ such that 
\begin{equation*}
\|f^0\|_{\mathcal{H}^{2n+2k}}\leq C_0\frac{e^{-\alpha c_0 t}}{A}.
\end{equation*}
\textit{Step 2: Control of $\eta^1$ and $\eta^2$.}
At this point some care with the constants is required. We denote $R_{\eta}$ the resolvent operator associated to 
$J\partial_x\delta^2H[S+\eta]$. Following the notation and result of theorem \ref{estimN}, we set
$C_e=\max_{n\leq p\leq n+k} C_{e,p}$. There exists a constant 
$C_R$ such that for $1\leq l\leq k$, under the conditions
\begin{equation}\label{petit1}
C_e(1/A+\|\eta\|_{X^{2n+2l}})^{1/4}\leq \alpha c_0/4,\ 
1/A+\|\eta\|_{X^{2n+2l}}\leq \min_{1\leq j\leq k} \varepsilon_{n+j},
\end{equation}
and using $te^t\leq e^{2t}$ then
\begin{equation}\label{resunif}
\|R_\eta(t,s)\|_{\mathcal{L}(\mathcal{H}^{2n+2l})}\leq C(1+|t-s|)e^{\alpha c_0|t-s|/4}\leq 
C_Re^{\alpha c_0|t-s|/2}.
\end{equation}
According to lemma \ref{taylor}, there exists $C_{Tayl}$ such that for 
\begin{eqnarray}
\label{condinfini}
j\leq n+k,\ \eta\in \mathcal{H}^{2j+2},U\in \mathcal{H}^{2j+2},
\ V=S+U, \
\|\eta\|_{L^\infty}+\|U\|_{L^\infty}\leq \inf S/2,\\
\label{taylunif}
\Rightarrow \|J\partial_x\delta H[V+\eta]-J\partial_x\delta H[V]-J\partial_x\delta^2H[V]\eta\|
_{\mathcal{H}^{2j}}\leq C_{Tayl}\|\eta\|_{\mathcal{H}^{2j+2}}^2.
\end{eqnarray}
With these notations we bound $\eta^1$ using the Duhamel formula, 
\begin{eqnarray*}
\eta^1&=&\int_t^\infty R_0(t,s)f^0(s)ds,\\
\Rightarrow 
\|\eta^1\|_{\mathcal{H}^{2n+2k}}&\leq& 
\int_t^\infty C_0C_Re^{\frac{\alpha c_0}{2}|t-s|}\frac{e^{-\alpha c_0 s}}{A}ds
= \frac{2C_0C_R}{A\alpha c_0}e^{-\alpha  c_0t}.
\end{eqnarray*}
We define $\delta =2C_0C_R/(A\alpha c_0)$, which can be as small as needed. \\
To bound $f^1=J\partial_x\delta H[S]+J\partial_x\delta H[S]\eta^1
-J\partial_x\delta H[S+\eta^1]$ we can use estimate 
\eqref{taylunif} (note that up to decreasing $\min_{1\leq j\leq k}
\varepsilon_{n+j}$, condition \eqref{petit1} is stronger than 
\eqref{condinfini})
\begin{equation*}
\|f^1\|_{\mathcal{H}^{2n+2(k-1)}}\leq C_{Tayl}\delta^2e^{-2\alpha c_0t}.
\end{equation*}
Next to use the resolvent estimate \eqref{resunif}, we need to bound $\|\eta^1\|_{X^{2n+2(k-1)}}$.
This is done thanks to a general estimate : if 
$\partial_t\eta=J\partial_x\delta H[S+U]\eta+f$, then using \eqref{borneL}
\begin{equation}\label{estdt}
\forall\,N\geq 2,
\|\partial_t\eta\|_{\mathcal{H}^{N-2}}\leq C_{X,N}(\|U\|_{\mathcal{H}^{N}})
\|\eta\|_{\mathcal{H}^{N}}+\|f\|_{\mathcal{H}^{N-2}}.
\end{equation}
We define $\di C_X=\max_{1\leq j\leq k}
C_{X,2n+2j}(\varepsilon_{n+j})$.
Since $\partial_t\eta^1=J\partial_x\delta^2H[S]\eta^1-f^0$, \eqref{estdt} gives
\begin{equation*}
\|\partial_t\eta^1\|_{\mathcal{H}^{2n+2(k-1)}}\leq C_X\|\eta^1\|_{\mathcal{H}^{2n+2k}}
+\frac{C_0e^{-\alpha c_0t}}{A}\leq C_X\delta+\frac{C_0}{A}.
\end{equation*}
In particular, $\|\eta^1\|_{X^{2n+2(k-1)}}\leq 
(C_X+1)\delta+C_0/A$. Therefore (up to increasing 
$A$) condition \eqref{petit1} is satisfied (with $l=k$), Duhamel's formula gives again
\begin{equation*}
\|\eta^2\|_{\mathcal{H}^{2n+2(k-1)}}\leq \int_t^\infty C_RC_{Tayl}\delta^2
e^{\alpha c_0|t-s|/2}e^{-2\alpha c_0s}ds
\leq \frac{C_RC_{Tayl}\delta}{\alpha c_0}\delta e^{-2\alpha c_0t}
:=q\delta e^{-2\alpha c_0t}.
\end{equation*}
We note for later use that for $A$ large enough, we have 
$q\leq 1/2$ so that $\sum_{j\geq 0}q^j\leq 1$.
\vspace{2mm}\\
\textit{Step 3. Induction.}
Assume we have constructed $(\eta^i)_{1\leq i\leq j}$ for some $j<k$, with 
\begin{equation*}
\|\eta^1\|_{\mathcal{H}^{2n+2k}}\leq \delta e^{-\alpha c_0t},\ 
\forall\,i\geq 2,\  \|\eta^i\|_{\mathcal{H}^{2n+2(k-i+1)}}\leq q^{2^{i-1}}\delta
e^{-2^{i-1}\alpha c_0t}.
\end{equation*}
In particular, $\|\sum_1^j\eta^i\|_{\mathcal{H}^{2n+2(k-i+1)}}\leq 2\delta$ and
from \eqref{taylunif} 
\begin{equation*}
\forall\, 1\leq i\leq j,\ \|f^i(t)\|_{\mathcal{H}^{2n+2(k-i)}}\leq C_{Tayl}q^{2^i}\delta^2
e^{-2^i\alpha c_0 t}.
\end{equation*}
Estimating $\partial_t\eta^i$ as for $\partial_t\eta^1$ we have 
\begin{equation*}
\bigg\|\partial_t(\sum_{i=1}^j\eta^i)\bigg\|_{\mathcal{H}^{2n+2(k-j)}}\leq \sum_{i=1}^jC_X
\|\eta^i\|_{\mathcal{H}^{2n+2(k-i+1)}}
+\|f^{i-1}\|_{\mathcal{H}^{2n+2(k-i)}}\leq 2C_X\delta +2C_{Tayl}\delta^2+\frac{C_0}{A},
\end{equation*}
therefore $\|\sum_1^j \eta^i\|_{X^{2n+2(k-j)}}\leq 2(C_X+1)\delta +2C_{Tayl}\delta^2+C_0/A$,
and the smallness conditions \eqref{petit1} are satisfied (with $l=k-j$) for $A$ large enough 
independent of $j$.\\
We can use the uniform resolvent estimate \eqref{resunif} and Taylor estimate \eqref{taylunif} 
as for the construction of $\eta^2$
\begin{eqnarray*}
\|\eta^{j+1}\|_{\mathcal{H}^{2n+2(k-j)}}&\leq& \int_t^\infty C_RC_{Tayl}q^{2^j}\delta^2
e^{\alpha c_0|t-s|/2}e^{-2^j\alpha c_0s}ds
\leq q^{2^j}\delta e^{-2^j\alpha c_0t}.
\end{eqnarray*}
By induction, we obtain $(\eta^j)_{1\leq j\leq k}$, the function $U^a\sum_1^k\eta^j$ is sufficient to end the proof since by construction and estimate \eqref{taylunif}
\begin{equation*}
\partial_t(S+\sum_1^k\eta^j)=J\partial_x\delta H[S+\sum_1^k\eta^j]+f^k,
\text{ with }\|f^k\|_{\mathcal{H}^{2n}}\leq C_{Tayl}q^{2^k}\delta e^{-2^kt},
\end{equation*}
so that the remainder $f^k$ is as small and rapidly decreasing as required for $k$ large enough.
\end{proof}
\section{Proof of the main result}\label{secpreuve}
This section is a compactness argument.
Let $V^a=\begin{pmatrix}\rho^a\\v^a\end{pmatrix}=S+U^a$ be an
approximate solution given by theorem \ref{solapp} with 
$U^a\in \mathcal{H}^{2n+2}$ and 
$n,\varepsilon,\ C_e$ to choose later.
Define $V^k$ the solution of \eqref{EK} 
with Cauchy data $V^k(k)=V^a(k)$. According to the energy estimate of proposition 
\ref{mainenergy}, we have for $\Delta U^k:=V^k-V^a$
\begin{equation*}
\bigg|\frac{d}{dt}\widetilde{\|}\Delta U^k\|_{\mathcal{H}^{2n}}\bigg|\leq C\big(\|\Delta U^k\|_{\mathcal{H}^{2n}}
+\|U^a\|_{\mathcal{H}^{2n+2}}+\|1/\rho^k+1/\rho^a\|_{L^\infty}\big)\big(
\widetilde{\|}\Delta U^k\|_{\mathcal{H}^{2n}}+\varepsilon e^{-C_et}\big).
\end{equation*}
Let $\di m=\inf_{(x,t)\in \R\times \R^+} \rho^a$, we pick $\varepsilon$ such that
$\widetilde{\|}\Delta U^k\|_{\mathcal{H}^{2n}}<2\varepsilon\Rightarrow \inf \rho^k\geq m/2$, 
and $\|\Delta U^k\|_{\mathcal{H}^{2n}}\leq C_1\widetilde{\|}\Delta U^k\|_{\mathcal{H}^{2n}}$.
Set 
$C=C(2C_1\varepsilon+\|U^a\|_{\mathcal{H}^{2n+2}}+3/m)$, and fix $C_e\geq 2C$. We can assume 
$C\geq 1$. Since $\Delta U^k(k)=0$, the energy estimate backwards in time gives 
(as long as $\|U(t)\|_{\mathcal{H}^{2n}}\leq 2\varepsilon$)
\begin{eqnarray*}
\frac{d}{dt}\big(e^{Ct}\widetilde{\|}\Delta U^k\|_{\mathcal{H}^{2n}}\big)\geq -\varepsilon e^{-(C_e-C)t}
\Rightarrow\widetilde{\|}\Delta U^k(t)\|_{\mathcal{H}^{2n}}^2\leq \frac{\varepsilon}{C_e-C}e^{-C_et}\leq \varepsilon e^{-C_et}.
\end{eqnarray*}
From a (backwards) continuation argument, the solution is well defined on $[0,k]$ for $\varepsilon$ small enough, 
and independently of $k$
$$\forall\,0\leq t\leq k,\ \|\Delta U^k(t)\|_{\mathcal{H}^{2n}}\leq C_1\varepsilon e^{-C_et}.$$ 
For $n$ large (actually $n=2$ is enough), we have from the equation 
$\partial_tV^k\in \mathcal{C}_b(\R^+,\mathcal{H}^{2n-2})$, 
indeed we recall that $\partial_x S$ is smooth and rapidly decaying:
\begin{eqnarray*}
\partial_t\rho^k=-\partial_x(\rho^ku^k), \text{ with }\rho^k\in (\rho_S+H^{2n+1}),\ u^k\in (u_S+H^{2n}),\text{ thus }
\partial_t\rho^k \in H^{2n-1} ,\\
\partial_tu^k=-\partial_x\bigg(g(\rho^k)+\underbrace{\frac{(u^k)^2}{2}}_{(u_S)^2+H^{2n}}-\underbrace{K\partial_x^2\rho^k}_{H^{2n-1}}
-\underbrace{\frac{1}{2}K'(\partial_x\rho^k)^2}_{H^{2n}}\bigg)\in H^{2n-2}.
\end{eqnarray*}
Similarly, $\partial_tU^a\in C_b(\R^+,\mathcal{H}^{2n})$. We deduce that $V^k-S=U^a+\Delta U^k$ is bounded in $C_b([0,k],\mathcal{H}^{2n})$ and $C^1_b([0,k],\mathcal{H}^{2n-2}$.
By weak* compactness, up to an extraction $V^k-S$ converges
weakly to some $U\in L^\infty(\R^+,\mathcal{H}^{2n})$. Moreover
for any bounded interval $J$, we have the compact embedding $\mathcal{H}^{2n-2}(J)\subset \mathcal{H}^{2n-3}(J)$,
so using the Ascoli-Arzela theorem, up to another extraction $V^k-S$ converges to $U$ in 
$C_{\text{loc}}(\R^+,\mathcal{H}_{\text{loc}}^{2n-2})$. For $2n-3\geq 2$ it is not hard to check that $S+U$ is a
solution of the Euler-Korteweg system \eqref{EK}.\\ 
Now due to the uniform estimate $\|V^k(t)-V^a(t)\|_{\mathcal{H}^{2n}}\leq \varepsilon e^{-C_et}\ (t\leq k)$, passing to the (weak) limit 
$$\|(S+U)(t)-V^a(t)\|_{\mathcal{H}^{2n}}\leq \varepsilon e^{-C_et}\ \text{(a.e.)}.$$
From theorem \ref{solapp}, we also know  $\|V^a-S\|_{\mathcal{H}^{2n}}\lesssim \frac{e^{-\alpha c_0t}}{A}$, therefore 
we can conclude 
$$\lim_{t\to \infty}\|U(t)-S(t)\|_{\mathcal{H}^{2n}}=0.$$
\begin{rmq}
A priori, the pointwise $\mathcal{H}^{2n}$ convergence holds 
only almost everywhere in 
$t$, however using the well-posedness theorem $1.1$ 
in \cite{BDD2}, one can prove that $U$ coincides with the 
$C(\R^+,\mathcal{H}^{2n})$ solution,
and by continuity the convergence holds for all $t$.
\end{rmq}
\appendix 
\section{Complements on traveling waves}\label{appendix}
\paragraph{Existence of kinks}
A traveling wave satisfy 
\begin{equation*}
 \left\{
 \begin{array}{lll}
  -c\partial_x \rho+\partial_x(\rho v)&=&0,\\
-c\partial_x v+\partial_x(v^2/2)+\partial_xg(\rho)&=&
\partial_x\bigg(K\partial_x^2\rho+\frac{1}{2}K'(\partial_x\rho)^2\bigg).
\end{array}
\right.
\end{equation*}
A first integration gives 
\begin{equation*}
 \left\{
 \begin{array}{lll}
  \rho(v-c)&=&j,\\
 \di  \frac{(v-c)^2}{2}+g(\rho)-K\partial_x^2\rho-\frac{1}{2}K'(\partial_x\rho)^2&=&q.
\end{array}
\right.
\end{equation*}
Assuming $\lim_{\pm \infty}\rho=\rho_\pm,\ \lim_{\pm \infty}
v= v_\pm$, we have
\begin{eqnarray}\label{masse}
j&=&\rho(v-c)=\rho_+(v_+-c)=\rho_-(v_--c),\\
\label{moment}
q&=&\frac{(v-c)^2}{2}+g(\rho)-K\partial_x^2\rho-\frac{1}{2}K'(\partial_x\rho)^2\\
&=&\frac{(v_+-c)^2}{2}+g(\rho_+)\\
&=&\frac{(v_--c)^2}{2}+g(\rho_-).
\end{eqnarray}
This implies
\begin{eqnarray}\label{edorho}
q=\frac{j^2}{2\rho^2}+g(\rho)-K\partial_x^2\rho
-\frac{1}{2}K'(\partial_x\rho)^2
\label{infoq}&=& \frac{j^2}{2\rho_+^2}+g(\rho_+)\\
&=&\frac{j^2}{2\rho_-^2}+g(\rho_-).
\end{eqnarray}
Set 
$f(\rho)= \frac{j^2}{2\rho^2}-q+g(\rho)$,  
we get two conditions 
\begin{equation}\label{condf}
f(\rho_+)=f(\rho_-)=0. 
\end{equation}
Multiplying \eqref{edorho} by $\partial_x\rho$ and integrating from $\rho_-$ to $\rho$
\begin{equation}\label{integrho}
\frac{1}{2}K(\partial_x\rho)^2=-q(\rho-\rho_-)
-\frac{j^2}{2}\bigg(\frac{1}{\rho}-\frac{1}{\rho_-}\bigg)
+G(\rho):=F(\rho),
\end{equation}
with $G$ the primitive of $g$ such that $G(\rho_-)=0$. 
From this integrated momentum equation 
we get one condition :
\begin{equation}\label{condF}
F(\rho_+)=0.
\end{equation}
This condition can be written only in term of $\rho_-,\rho_+$:
\begin{equation}\label{conditionintegrale}
\frac{G(\rho_+)-G(\rho_-)}{\rho_--\rho_-}=\frac{g(\rho_+)\rho_++g(\rho_-)\rho_-}{\rho_++\rho_-}.
\end{equation}
Lastly according to \eqref{moment} $\rho$ satisfies the following system of ODE 
\begin{equation*}
\left\{
\begin{array}{lll}
\sqrt{K}\partial_x\rho&=&w\\
\sqrt{K}\partial_xw&=&j^2/2\rho^2+g(\rho)-q,
\end{array}
\right.
\end{equation*}
up to a change of variable it is hamiltonian (with energy 
$F(\rho)$) therefore steady states can only be centers or saddles,
and a traveling wave connects two saddle points. 
So $(\rho_\pm,0)$ should be a
saddle point, which leads to a last condition: 
the characteristic equation at $(\rho_\pm,0)$ is 
\begin{equation*}
\lambda ^2+j^2/\rho_\pm^3-g'(\rho_\pm)=0,
\end{equation*}
and the roots in $\lambda$ are real with opposite sign under the 
condition
\begin{equation}
\label{condj}
j^2<\rho_\pm^3g'(\rho_\pm)\Leftrightarrow (v_\pm-c)^2<\rho_\pm g'(\rho_\pm)\Leftrightarrow f'(\rho_\pm)>0.
\end{equation}
(we will see several interpretations of this condition).
Conversely, assume \eqref{condf},\eqref{condF},\eqref{condj} are satisfied, and 
that $f$ only changes sign once on $(\rho_-,\rho_+)$. Due to \eqref{condf},
\eqref{condj}, $f'(\rho_{\pm}>0$ thus $f$ must be 
positive then negative on $(\rho_-,\rho_+)$, and from \eqref{condF}, 
$F$ remains positive on $(\rho_-,\rho_+)$, but vanishes at second order at $\rho_\pm$. 
The existence of a kink then just follows from the integration of 
$\pm\sqrt{K}\partial_x\rho/\sqrt{2F(\rho)}=1$ (with a choice of sign adapted to the one of $\rho_--\rho_+$).
To summarize, provided this equation is satisfied $c$ is a free parameter, and 
either $\rho_+$ or $\rho_-$ is used to fully parametrize the traveling waves. Kinks should 
thus form locally two dimensional manifolds.
\begin{rmq}
As the construction of the profile $\rho$ depends on $(v_+-c)^2$, we can assume $c-v_+>0$.
\end{rmq}

\paragraph{The speed of kinks, some geometry}
The momentum equation is 
\begin{equation*}
\frac{j^2}{2\rho^2}-\frac{j^2}{2\rho_+^2}+g(\rho)=K\partial_x^2\rho
+\frac{1}{2}K'(\partial_x\rho)^2,
\end{equation*}
\begin{equation}
\Rightarrow -\frac{j^2(\rho-\rho_+)^2}{2\rho\rho_+^2}+G(\rho)
=\frac{1}{2}K(\partial_x\rho)^2\geq 0.
\end{equation}
Letting $x\to \pm\infty$, from the sign condition we find again \eqref{condj}
\begin{equation}\label{condspeed}
\frac{j^2}{\rho_\pm^3}=\frac{(v_\pm-c)^2}{\rho_\pm}\leq g'(\rho_\pm).
\end{equation}
This inequality gives a geometric interpretation of \eqref{condf}, that we rewrite
\begin{equation*}
g(\rho_-)=\frac{-j^2}{2\rho_-^2}+q,\ g(\rho_+)=\frac{-j^2}{2\rho_+^2}+q,
\end{equation*}
meaning that $\rho_\pm$ are intersection points of the curves $g, -j^2/2\rho^2+q$, and conditions 
\eqref{condj} mean that the curves  intersect transversally at $\rho_\pm$. 
Condition \eqref{condF} means that 
the total signed area between the two curves from $\rho_-$ to $\rho_+$ must be zero. 
See figure \ref{figwaals}.
When $g$ follows a Van Der Waals law, such conditions can be met we refer to \cite{BDD2} for some relevant examples.
\begin{figure}[h]\label{figwaals}
\begin{center}
 \includegraphics[scale=0.5]{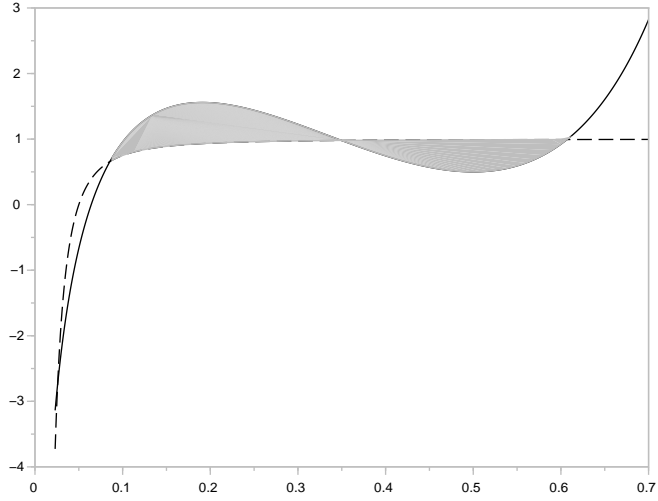}
 \end{center}
 \caption{Full line $g$, dashed line $-j^2/(2\rho^2)+q$, the two shaded areas should be equal.}
\end{figure}

\paragraph{The dimension of families of kinks.} There exists a kink provided equations \eqref{condf}, \eqref{condF} are met (\eqref{condj} is open and therefore plays no role 
for the dimension), namely
\begin{equation*}
 f(\rho_+)=f(\rho_-)=\int_{\rho_-}^{\rho_+}f(\rho)d\rho=0,\text{ where $f$ depends on }\rho_\pm,c,j,q.
\end{equation*}
Consider the application $\varphi:(\rho_\pm,j,q,c)\to \begin{pmatrix} f(\rho_+)\\f(\rho_-)\\
\int_{\rho_-}^{\rho_+}f(\rho)d\rho\end{pmatrix}$, we have 
\begin{eqnarray*}
D\varphi &=&
\begin{pmatrix}
-j^2/\rho_-^3+g'(\rho_-) & 0 & j/\rho_-^2 & -1 & 0\\
0 & -j^2/\rho_+^3+g'(\rho_+) & j/\rho_+^2 & -1 & 0\\
-f(\rho_-) & f(\rho_+) & \frac{j^2}{2}(1/\rho_--1/\rho_+) & \rho_--\rho_+ & 0
\end{pmatrix}\\
&=&\begin{pmatrix}
-j^2/\rho_-^3+g'(\rho_-) & 0 & j/\rho_-^2 & -1 & 0\\
0 & -j^2/\rho_+^3+g'(\rho_+) & j/\rho_+^2 & -1 & 0\\
0 & 0 & \frac{j^2}{2}(1/\rho_--1/\rho_+) & \rho_--\rho_+ & 0
\end{pmatrix}.
\end{eqnarray*}
According to the sign condition \eqref{condj}, in the generic case 
$-j^2/\rho_\pm^3+g'(\rho_\pm)>0$, so the rank of the matrix is three and the
kinks form a manifold of dimension two.

\paragraph{The case of solitons} Kinks can not provide
a nontrivial soliton in the limit $\rho_-\to \rho_+$, indeed kinks are monotonous 
therefore the ``soliton'' limit of a kink is actually a constant solution.
Nevertheless, the construction of solitons follows the same lines. We denote $\rho_+
=\lim_{\pm \infty}\rho$.
Since $g$ is a primitive of $g'$, we can assume $g(\rho_+)=0$. 
Equation \eqref{condf} gives
\begin{equation*}
f(\rho_+)=\frac{j^2}{2\rho_+^2}-q=0\Rightarrow f(\rho)=\frac{j^2}{2\rho^2}
-\frac{j^2}{2\rho_+^2}+g(\rho).
\end{equation*}
Then \eqref{condF} is free so 
\begin{equation}\label{intmom2}
\frac{1}{2}K(\rho')^2=\frac{-j^2}{2\rho\rho_+^2}(\rho_+-\rho)^2+G(\rho)
=\frac{-(c-v_+)^2}{2\rho}(\rho_+-\rho)^2+G(\rho).
\end{equation}
For $j^2<\rho_+^3g'(\rho_+)\Leftrightarrow (v_+-c)^2<\rho_+g'(\rho_+)$, we can 
define 
$$\di\rho_m=\sup\bigg\{\rho<\rho_+:\ \frac{-j^2}{2\rho\rho_+^2}(\rho_+-\rho)^2
+G(\rho)=0\bigg\}.$$
From basic ODE arguments, there exists a homoclinic orbit to $\rho_+$ with minimal 
value $\rho_m$; a ``bubble'' decreasing from $\rho_+$ to $\rho_m$ then increasing back 
to $\rho_+$. 
\begin{rmq}\label{compaspeed}
We recall that a kink of speed $c_k$ and right endstate $(\rho_+,v_+)$ 
satisfies $(v_+-c_k)^2<\rho_+g'(\rho_+)$, and since its construction depends 
on $(v_+-c_k)^2$ rather than $v_+-c_k$, we may assume $c_k-v_+\geq 0$.
In particular since there exists solitons of speed $c_s$ with $(v_+-c_s)^2$ 
arbitrarily close to $\rho_+g'(\rho_+)$, there always exists solitons faster than the kink and sharing the same endstate. 
\end{rmq}

\paragraph{Existence of kink-stable solitons configuration}
According to remark \ref{compaspeed}, given a kink with right endstate $(\rho_+,v_+)$, 
there exists solitons with same endstate and larger speed satisfying $c-v_+>0$. We are left to check wether such solitons are stable.\\
We assume here that the asymptotic state 
$(\rho_+,v_+)$ is fixed, so that solitons only depend on the speed $c$,
and we also assume $g''(\rho_+)\geq 0$ (this is true for the Van Der Waals case).\\
For consistency, we first prove that the stability condition $dP/dc<0$ 
is indeed equivalent to the stability condition of Benzoni et 
al\cite{BDD2}. To do so, we recall 
the definition of momentum of instability from \cite{BDD2}.
The equations satisfied by a soliton are
\begin{equation*}
\left\{
\begin{array}{lll}
 -c(v-v_+)+v^2/2+g(\rho)-K\partial_x^2\rho-\frac{1}{2}K'(\partial_x\rho)^2&=&
 v_+^2/2+g(\rho_+),\\
  -c(\rho-\rho_+)+\rho v&=&\rho_+v_+.
\end{array}
\right.
\end{equation*}
Defining $\di H=\frac{1}{2}\int \rho v^2-\rho_+v_+^2+K(\partial_x\rho)^2+2G(\rho)dx$, and
recalling $P=\int (\rho-\rho_+) (v-v_+)$, they can be expressed in an abstract way 
\begin{equation}\label{abstract}
\delta H-c\delta P=\big(u_+^2/2+g(\rho_+)\big)\delta P_1+\rho_+v_+\delta P_2:=\lambda_1\delta P_1+\lambda_2\delta P_2,
\end{equation}
where $P_1=\int \rho-\rho_+dx,\ P_2=\int v-v_+dx$.
The momentum of instability is then 
\begin{equation}
m(c)=H-cP-\lambda_1 P_1-\lambda_2P_2,
\end{equation}
and the stability condition of \cite{BDD2} is $m''(c)>0$. 
\begin{lemma}
 The condition $m''(c)>0$ is equivalent to 
 \begin{equation}\label{condstab}
 \frac{dP}{dc}=
  \frac{d}{dc}\int_\R \frac{(\rho-\rho_+)^2}{\rho}(c-v_+)dx<0.
 \end{equation}
\end{lemma}
\begin{proof}
Denote $'$ the derivative with respect to $c$, using \eqref{abstract} we have
\begin{equation}
m'(c)=H'-cP'-P-\lambda_1P'_1-\lambda_2P'_2=-P,
\end{equation}
We differentiate again and use the identity  $\rho(v-c)=\rho_+(v_+-c)$
\begin{eqnarray*}
m''(c)=-P' &=&-\frac{d}{dc}\int_\R (\rho-\rho_+)(v-v_+)dx\\
&=&-\frac{d}{dc}\int_\R \frac{(\rho-\rho_+)^2}{\rho}(c-v_+)dx.
\end{eqnarray*}
The condition $m''>0$ gives the expected result.
\end{proof}
The so-called transonic limit corresponds to $j^2/\rho_+^2=(v_+-c)^2\to 
\rho_+g'(\rho_+),$ so we set $j^2=\rho_+^3g'(\rho_+)(1-\varepsilon)$. From numerical 
computations it was conjectured in \cite{BDD2} that solitons are stable in the transonic 
limit, and this is rigorously proved with the following result.
As it gives the existence of stable solitons with 
speed arbitrarily close to $\sqrt{\rho_+g'(\rho_+)}$, it also 
provides the existence of kink-stable soliton configurations.
\begin{lemma}
For $\varepsilon$ small enough and $g''(\rho_+)>0$, 
``bubble'' solitons of speed $\sqrt{\rho_+g'(\rho_+)(1-\varepsilon)}$
are stable.
\end{lemma}

\begin{proof} The condition $dP/dc<0$ is equivalent to 
$dP/d\varepsilon>0$ and 
equation \eqref{intmom2} reads 
\begin{eqnarray*}
\frac{1}{2}K(\partial_x\rho)^2&=&(\rho-\rho_+^2)\bigg(\frac{g'(\rho_+)}{2}+\frac{g''(\rho_+)(\rho-\rho_+)}{6}
-\frac{\rho_+g'(\rho_+)(1-\varepsilon)}{2\rho}+O(\rho-\rho_+)^2\bigg)\\
&=&(\rho-\rho_+)^2\bigg(\frac{\varepsilon\rho_+g'(\rho_+)}{2\rho}+
\bigg(\frac{g''(\rho_+)}{6}+\frac{g'(\rho_+)}{2\rho}\bigg)(\rho-\rho_+)+O(\rho-\rho_+)^2\bigg)\\
&:=&\frac{\rho_+g'(\rho_+)}{2\rho}(\rho-\rho_+)^2\big(\varepsilon+\alpha(\rho-\rho_+)\big)+O(\rho-\rho_+)^4.
\end{eqnarray*}
Note that $\alpha>0$, in the limit $\varepsilon\to 0+$, solitons have an amplitude $\rho_+-\rho_m\sim \varepsilon/\alpha\to 0$, where 
$\rho_m(\varepsilon)$ is the minimum of $\rho$, and in this regime 
$\rho_m'(\varepsilon)<0$.
Up to translation, we can assume that the minimum of $\rho$ is reached at $x=0$, and $\rho$ is strictly decreasing on $(-\infty,0)$.
Using on $x\in (0,\infty)$ 
the change of variable $\rho(x)=\rho,\ dx=\sqrt{\frac{K}{2F}}d\rho$ we find 
\begin{equation}
 P=2\int_{\rho_m}^{\rho_+}\frac{(\rho-\rho_+)^2(c-v_+)}{\rho}\sqrt{\frac{K}{2F}}d\rho.
\end{equation}
As is expectable, the situation is somewhat degenerate at $\varepsilon=0$, as one can 
check that $P(c)=P(\sqrt{\rho_+g'(\rho_+)(1-\varepsilon)})=O(\varepsilon^{3/2})$.
This is handled by a factorization of $F$ (see \eqref{integrho}): 
\begin{equation*}
F=\frac{\rho_+g'(\rho_+)(\rho-\rho_+)^2}{2\rho}\bigg(\varepsilon+\frac{2\rho G
}{\rho_+g'(\rho_+)
(\rho-\rho_+)^2 }-1\bigg):=
\frac{\rho_+g'(\rho_+)(\rho-\rho_+)^2}{2\rho}\big(\varepsilon+H(\rho)\big).
\end{equation*}
Here $H(\rho_+)=0$ and by construction $H(\rho_m(\varepsilon))+\varepsilon=0$. 
The condition $\alpha>0$ implies $H'(\rho_+)>0$, so 
$\varphi(\rho,\varepsilon):=(H+\varepsilon)/(\rho-\rho_m)$ is well defined 
near $(\rho,\varepsilon)=(\rho_+,0)$, smooth and does not cancel. To summarize,
$
F=\frac{\rho_+g'(\rho_+)(\rho-\rho_+)^2(\rho-\rho_m)}{2\rho}\varphi(\rho,\varepsilon)
$. Denoting $\delta(\varepsilon)=\rho_+-\rho_m$, we use the change of variables 
$\rho=\rho_+-\delta r$ : 
\begin{eqnarray*}
 P&=&2\int_{0}^{1}\frac{\delta^2r^2(c-v_+)}{\rho}\sqrt{\frac{\rho K}{\rho_+g'(\rho_+)\delta^2r^2\delta(1-r)\varphi(r,\varepsilon)}}\delta dr\\
&=& 2\int_{0}^{1}\frac{\delta^{3/2}r(c-v_+)}{\rho}\sqrt{\frac{\rho K}{\rho_+g'(\rho_+)(1-r)\varphi(r,\varepsilon)}}dr.
\end{eqnarray*}
From $\rho_m'(\varepsilon)<0$, $\varepsilon\to \delta(\varepsilon)$ is locally invertible 
and the stability condition is equivalent to $dP/d\delta>0$, but it is clear from the 
formula that 
$$dP/d\delta=\frac{3\delta^{1/2}}{2}\int_0^1\frac{r(c-v_+)}{\rho}
\sqrt{\frac{\rho K}{\rho_+g'(\rho_+)(1-r)\varphi(r,\varepsilon)}} dr+
O(\delta^{3/2}),
$$
which is positive for $\delta$ small enough.
\end{proof}
\bibliography{biblio}
 \bibliographystyle{plain}
\end{document}